\documentclass{article}

\usepackage{subeqnarray}
\usepackage{bm}
\usepackage{tikz}
\usepackage{dsfont}
\usepackage{upgreek}
\usepackage{booktabs}
\usepackage{fancyhdr}
\usepackage{multirow}
\usepackage{multicol}
\usepackage{float}
\usepackage[inline]{enumitem}   

\usepackage[ruled,vlined,linesnumbered,inoutnumbered]{algorithm2e}

\usepackage{bbding}
\usepackage{caption}
\usepackage{todonotes}

\usepackage{mathrsfs}

\usepackage{amsmath}
\usepackage{amsthm}
\usepackage{amssymb}
\usepackage{amsfonts}

\usepackage{lineno}

\usepackage[top=2.5cm, bottom=2.5cm, left=3cm, right=3cm]{geometry} 

\usepackage{graphics,graphicx,epsf,epstopdf,subfigure}
\graphicspath{{./Figures/}}
\ifpdf
\DeclareGraphicsExtensions{.eps,.pdf,.png,.jpg}
\else
\DeclareGraphicsExtensions{.eps}
\fi

\providecommand{\U}[1]{\protect\rule{.1in}{.1in}}

\usepackage[breaklinks,colorlinks,linkcolor=black,citecolor=black,urlcolor=black]{hyperref}

\newtheorem{theorem}{Theorem}[section]

\newtheorem{lemma}[theorem]{Lemma}
\newtheorem{proposition}[theorem]{Proposition}
\newtheorem{definition}[theorem]{Definition}

\newtheorem{assumption}{Assumption}

\numberwithin{equation}{section}

\newcommand{\tr}{\mathrm{tr}}

\newcommand{\st}{\mathrm{s.\,t.}\,\,}

\newcommand{\bR}{\mathbb{R}}
\newcommand{\bN}{\mathbb{N}}

\newcommand{\cL}{\mathcal{L}}
\newcommand{\cB}{\mathcal{B}}
\newcommand{\cR}{\mathcal{R}}

\newcommand{\tG}{\mathtt{G}}
\newcommand{\tV}{\mathtt{V}}
\newcommand{\tE}{\mathtt{E}}

\newcommand{\Rnp}{\mathbb{R}^{n \times p}} 
 
\newcommand{\Rpp}{\mathbb{R}^{p \times p}} 
\newcommand{\Rnn}{\mathbb{R}^{n \times n}}  
\newcommand{\Rnm}{\mathbb{R}^{n \times m}}
\newcommand{\Rmn}{\mathbb{R}^{m \times n}}

\newcommand{\barX}{\bar{X}}

\newcommand{\bfX}{\mathbf{X}}
\newcommand{\bfD}{\mathbf{D}}

\newcommand{\Xik}{X_{i,k}}
\newcommand{\Xjk}{X_{j,k}}
\newcommand{\Xikn}{X_{i,k+1}}

\newcommand{\Xikp}{X_{i,k-1}}

\newcommand{\Dik}{D_{i,k}}
\newcommand{\Djk}{D_{j,k}}
\newcommand{\Dikn}{D_{i,k+1}}

\newcommand{\Dikp}{D_{i,k-1}}

\newcommand{\bfXk}{\mathbf{X}_{k}}
\newcommand{\bfXkn}{\mathbf{X}_{k+1}}

\newcommand{\avXk}{\bar{\mathbf{X}}_{k}}
\newcommand{\avXkn}{\bar{\mathbf{X}}_{k+1}}

\newcommand{\avDk}{\bar{\mathbf{D}}_{k}}
\newcommand{\avDkn}{\bar{\mathbf{D}}_{k+1}}

\newcommand{\avHk}{\bar{\mathbf{H}}_{k}}
\newcommand{\avHkn}{\bar{\mathbf{H}}_{k + 1}}

\newcommand{\bfDk}{\mathbf{D}_{k}}
\newcommand{\bfDkn}{\mathbf{D}_{k+1}}

\newcommand{\bfHk}{\mathbf{H}_{k}}
\newcommand{\bfHkn}{\mathbf{H}_{k + 1}}

\newcommand{\bfone}{\mathbf{1}}
\newcommand{\bfW}{\mathbf{W}}
\newcommand{\bfJ}{\mathbf{J}}

\newcommand{\barXk}{\bar{X}_{k}}
\newcommand{\barXkn}{\bar{X}_{k + 1}}

\newcommand{\barDk}{\bar{D}_{k}}

\newcommand{\barHk}{\bar{H}_{k}}

\newcommand{\zz}{^{\top}}

\newcommand{\ff}{_{\mathrm{F}}}
\newcommand{\fs}{^2_{\mathrm{F}}}

\newcommand{\Snp}{\mathcal{S}_{n,p}}

\newcommand{\dkh}[1]{\left(#1\right)}
\newcommand{\hkh}[1]{\left\{#1\right\}}

\newcommand{\jkh}[1]{\left\langle#1\right\rangle}
\newcommand{\norm}[1]{\left\|#1\right\|}
\newcommand{\abs}[1]{\left\lvert #1\right\rvert}

\newcommand{\sym}{\mathrm{sym}}

\newcommand{\grad}{\mathrm{grad}\,}

\newcommand{\iid}{i \in [d]}

\newcommand{\iin}{i \in [n]}
\newcommand{\sumiid}{\sum\limits_{i=1}^d}

\newcommand{\sumjjd}{\sum\limits_{j=1}^d}

\allowdisplaybreaks
\graphicspath{ {./results/} }
\definecolor{Gray}{rgb}{0.5,0.5,0.5}

\makeatletter

\newcommand{\Rmnum}[1]{\expandafter\@slowromancap\romannumeral #1@}
\makeatother

\begin{document}

\title{Decentralized Optimization Over the Stiefel Manifold 
	by an Approximate Augmented Lagrangian Function}

\author{
	Lei Wang\thanks{State Key Laboratory of Scientific and Engineering 
		Computing, Academy of Mathematics and Systems Science, Chinese  Academy of 
		Sciences, and University of Chinese Academy of Sciences, China 
		(\href{mailto:wlkings@lsec.cc.ac.cn}{wlkings@lsec.cc.ac.cn}). 
		Research is supported by the National Natural Science Foundation of China 
		(No. 11971466 and 11991020).}
	\and 
	Xin Liu\thanks{State Key Laboratory of Scientific and Engineering 
		Computing, Academy of Mathematics and Systems Science, Chinese Academy of 
		Sciences, 
		and University of Chinese Academy of Sciences, China 
		(\href{mailto:liuxin@lsec.cc.ac.cn}{liuxin@lsec.cc.ac.cn}). 
		Research is supported in part by the National Natural Science Foundation of 
		China (No. 12125108 and 11991021), 
		Key Research Program of Frontier Sciences, 
		Chinese Academy of Sciences (No. ZDBS-LY-7022).}
}

\date{} 
\maketitle

\begin{abstract}
	
	In this paper, we focus on the decentralized optimization problem over the Stiefel manifold,
	which is defined on a connected network of $d$ agents.
	The objective is an average of $d$ local functions,
	and each function is privately held by an agent and encodes its data.
	The agents can only communicate with their neighbors 
	in a collaborative effort to solve this problem.
	In existing methods, 
	multiple rounds of communications are required to guarantee the convergence,
	giving rise to high communication costs.
	In contrast, this paper proposes a decentralized algorithm, called DESTINY,
	which only invokes a single round of communications per iteration.
	DESTINY combines gradient tracking techniques 
	with a novel approximate augmented Lagrangian function.
	The global convergence to stationary points is rigorously established.
	Comprehensive numerical experiments demonstrate that DESTINY
	has a strong potential to deliver a cutting-edge performance
	in solving a variety of testing problems. 
	
\end{abstract}


\section{Introduction}

	This paper is dedicated to providing an efficient approach to
	the decentralized optimization with orthogonality constraints,
	a problem defined on a connected network and solved by $d$ agents cooperatively.
	\begin{equation}\label{eq:opt-stiefel}
		\begin{aligned}
			\min\limits_{X \in \Rnp} \hspace{2mm} 
			& f(X) := \dfrac{1}{d} \sumiid f_i(X) \\
			\st \hspace{3mm} & X\zz X = I_p,
		\end{aligned}
	\end{equation}
	where $I_p$ is the $p \times p$ identity matrix,
	and $f_i: \Rnp \to \bR$ is a local function privately known by the $i$-th agent.
	The set of all $n \times p$ orthogonal matrices 
	is referred to as Stiefel manifold \cite{Stiefel1935},
	which is denoted by $\Snp = \{X \in \Rnp \mid X\zz X = I_p\}$.
	Throughout this paper, 
	we make the following blanket assumptions about these local functions.
	\begin{assumption}\label{asp:smooth}
		Each local objective function $f_i$ is smooth 
		and its gradient $\nabla f_i$ is locally Lipschitz continuous.
	\end{assumption}

	We consider the scenario that the agents can only exchange information 
	with their immediate neighbors through the network,
	which is modeled as a connected undirected graph $\tG = (\tV, \tE)$.
	Here, $\tV = [d] := \{1, 2, \dotsc, d\}$ is composed of all the agents
	and $\tE = \{(i, j) \mid i \text{~and~} j \text{~are connected}\}$ 
	denotes the set of all the communication links.
	Let $X_i \in \Rnp$ be the local copy of the common variable $X$ kept by the $i$-th agent.
	Then the optimization problem \eqref{eq:opt-stiefel} over the Stiefel manifold
	can be equivalently recast as the following decentralized formulation.
	\begin{subequations}\label{eq:opt-dest}
		\begin{align}
			\label{eq:obj-dest}
			\min\limits_{ \{X_i\}_{i=1}^d } \hspace{2mm} 
			& \dfrac{1}{d} \sumiid f_i(X_i) \\
			\label{eq:con-consensus}
			\st \hspace{2.5mm} &  X_i = X_j, \hspace{2mm} (i, j) \in \tE, \\
			\label{eq:con-orth}
			& X_i \in \Snp, \hspace{2mm} \iid.
		\end{align}
	\end{subequations}
	Since the graph $\tG$ is assumed to be connected,
	the constraints \eqref{eq:con-consensus} enforce the consensus condition
	$X_1 = X_2 = \dotsb = X_d$,
	which guarantees the equivalence between \eqref{eq:opt-stiefel} and \eqref{eq:opt-dest}.
	Due to the nonconvexity of Stiefel manifolds,
	it is challenging to design efficient algorithms to solve \eqref{eq:opt-dest}.

\subsection{Broad applications}\label{subsec:broad-applications}

	Problems of the form \eqref{eq:opt-stiefel} that require decentralized computations 
	are found widely in various scientific and engineering areas.
	We briefly describe several representative applications of it below.
	
	\paragraph{Application 1: Principal Component Analysis.}
	The principal component analysis (PCA) is a basic and ubiquitous tool for dimensionality reduction 
	serving as a critical procedure or preprocessing step 
	in numerous statistical and machine learning tasks.
	Mathematically, PCA amounts to solving the following optimization problem.
	\begin{equation}\label{eq:opt-pca}
		\begin{aligned}
			\min\limits_{X \in \Rnp} \hspace{2mm} 
			& -\frac{1}{2d} \sumiid \tr \dkh{ X\zz A_i A_i\zz X} \\
			\st \hspace{3mm} & X \in \Snp,
		\end{aligned}
	\end{equation}
	where $A = [A_1, \dotsc, A_d] \in \Rnm$ is the data matrix 
	consisting of $m$ samples with $n$ features.
	In the decentralized setting, 
	$A_i \in \bR^{n \times m_i}$ represents the local data stored in the $i$-th agent
	such that $m = m_1 + \dotsb + m_d$.
	This is a natural scenario since each sample is a column of $A$.

	\paragraph{Application 2: Orthogonal Least Squares Regression.}
	The orthogonal least squares regression (OLSR) proposed in \cite{Zhao2016}
	is a supervised learning method for feature extraction in linear discriminant analysis.
	Suppose we have $m$ samples with $n$ features drawn from $p$ classes.
	Then OLSR identifies an orthogonal transformation $X \in \Snp$ by solving the following problem.
	\begin{equation}\label{eq:opt-olsr}
		\begin{aligned}
			\min\limits_{X \in \Rnp} \hspace{2mm} 
			& \frac{1}{d} \sumiid \norm{C_i\zz X - D_i\zz}\fs \\
			\st \hspace{3mm} & X \in \Snp,
		\end{aligned}
	\end{equation}
	where $C = [C_1, \dotsc, C_d] \in \Rnm$
	and $D = [D_1, \dotsc, D_d] \in \bR^{p \times m}$ 
	are two matrices constructed by samples.
	Under the decentralized scenario,
	$C_i \in \bR^{n \times m_i}$ and $D_i \in \bR^{p \times m_i}$ 
	are local data and corresponding class indicator matrix 
	stored in the $i$-th agent, respectively.
	Here, $m = m_1 + \dotsb + m_d$.
	As claimed in \cite{Zhao2016},
	enforcing the orthogonality on transformation matrices 
	is conducive to preserve local structure 
	and discriminant information.
	
	\paragraph{Application 3: Sparse Dictionary Learning.}
	Sparse dictionary learning (SDL) is a classical unsupervised representation learning method,
	which finds wide applications in signal and image processing 
	due to its powerful capability of exploiting low-dimensional structures in high-dimensional data.
	Motivated by the fact that maximizing a high-order norm 
	promotes spikiness and sparsity simultaneously,
	the authors of \cite{Zhai2020} introduce the following $\ell_4$-norm maximization problem
	with orthogonality constraints for SDL.
	\begin{equation}\label{eq:opt-sdl}
		\begin{aligned}
			\min\limits_{X \in \Rnp} \hspace{2mm} 
			& -\frac{1}{4d} \sumiid \norm{X\zz B_i}_4^4 \\
			\st \hspace{3mm} & X \in \Snp,
		\end{aligned}
	\end{equation}
	where $B = [B_1, \dotsc, B_d] \in \Rnm$ is the data matrix 
	consisting of $m$ samples with $n$ features,
	and the $\ell_4$-norm is defined as $\norm{Y}_4^4 = \sum_{i, j} Y(i, j)^4$.
	In the decentralized setting, 
	$B_i \in \bR^{n \times m_i}$ is the local data stored in the $i$-th agent
	such that $m = m_1 + \dotsb + m_d$.

\subsection{Literature survey}

	Recent years have witnessed the development of algorithms 
	for optimization problems over the Stiefel manifold, 
	including gradient descent approaches 
	\cite{Manton2002,Nishimori2005,Abrudan2008,Absil2008},
	conjugate gradient approaches \cite{Edelman1998,Sato2016dai,Zhu2017riemannian},
	constraint preserving updating schemes \cite{Wen2013,Jiang2015},
	Newton methods \cite{Hu2018,Hu2019},
	trust-region methods \cite{Absil2006},
	multipliers correction frameworks \cite{Gao2018,Wang2021multipliers},
	augmented Lagrangian methods \cite{Gao2019},
	sequential linearized proximal gradient algorithms \cite{Xiao2021penalty},
	and so on.
	Moreover, exact penalty models are constructed for this special type of problems,
	such as PenC \cite{Xiao2020} and ExPen \cite{Xiao2021solving}.
	Efficient infeasible algorithms can be designed based on these two models.
	Unfortunately, none of the above algorithms take the decentralized optimization into account.
	
	Recently, several distributed ADMM algorithms designed for centralized networks 
	have emerged to solve various variants of PCA problems
	\cite{Wang2020distributed,Wang2021communication},
	which pursue a consensus on the subspaces spanned by splitting variables.
	This strategy relaxes feasibility restrictions and significantly improves the convergence.
	However, these algorithms can not be straightforwardly extended to the decentralized setting.

	Decentralized optimization in the Euclidean space
	has been adequately investigated in recent decades.
	Most existing algorithms adapt classical Euclidean methods to the decentralized setting.
	For instance, in the decentralized gradient descent (DGD) method \cite{Nedic2009},
	each agent combines the average of its neighbors with a local negative gradient step.
	DGD can only converge to a neighborhood of a solution with a constant stepsize,
	and it has to employ diminishing stepsizes to obtain an exact solution 
	at the cost of slow convergence \cite{Yuan2016,Zeng2018}.
	To cope with this speed-accuracy dilemma of DGD,
	EXTRA \cite{Shi2015} corrects the convergence error 
	by introducing two different mixing matrices 
	(to be defined in Subsection \ref{subsec:preliminaries})
	for two consecutive iterations and leveraging their difference to update local variables.
	And decentralized gradient tracking (DGT) algorithms 
	\cite{Xu2015,Qu2017,Nedic2017}
	introduce an auxiliary variable to maintain a local estimate 
	of the global gradient descent direction.
	Both EXTRA and DGT permit the use of a constant stepsize 
	to attain an exact solution without sacrificing the convergence rate.
	The above three types of algorithms are originally tailored for unconstrained problems.
	Their constrained versions, such as \cite{Ram2010,Di2016,Scutari2019},
	are still incapable of solving \eqref{eq:opt-dest} 
	since they can only deal with convex constraints.
	Another category of methods is based on the primal-dual framework,
	including decentralized augmented Lagrangian methods \cite{Hong2017,Hajinezhad2019}
	and decentralized ADMM methods \cite{Chang2015multi,Ling2015,Liu2019c}.
	In the presence of nonconvex constraints,
	these algorithms are not applicable to \eqref{eq:opt-dest} either.

	As a special case of \eqref{eq:opt-dest},
	a considerable amount of effort has been devoted to 
	developing algorithms for decentralized PCA problems.
	For instance, a decentralized ADMM algorithm is proposed in \cite{Schizas2015}
	based on the low-rank matrix approximation model.
	And \cite{Gang2019} develops a decentralized algorithm called DSA
	and its accelerated version ADSA by adopting a similar update formula of DGD 
	to an early work of training neural networks \cite{Sanger1989}.
	The linear convergence of DSA to a neighborhood of a solution
	is established in \cite{Gang2022linearly},
	but ADSA does not yet have theoretical guarantees.
	Moreover, \cite{Gang2021fast} further introduces a decentralized version of
	an early work for online PCA \cite{Krasulina1970method},
	which enjoys the exact and linear convergence.
	Recently, \cite{Ye2021} combines gradient tracking techniques with subspace iterations
	to develop an exact algorithm DeEPCA with a linear convergence rate.
	Generally speaking, 
	it is intractable to extend these algorithms to generic cases.
	
	In a word, all the aforementioned approaches 
	can not be employed to solve \eqref{eq:opt-dest}.
	In fact, the investigation of decentralized optimization over the Stiefel manifold 
	is relatively limited.
	To the best of our knowledge, there is only one work \cite{Chen2021decentralized}
	targeting at this specific type of problem,
	which extends DGD and DGT from the Euclidean space to the Stiefel manifold.
	The resulting algorithms are called DRSGD and DRGTA, respectively.
	Analogous to the Euclidean setting,
	DRGTA can achieve the exact convergence with constant stepsizes,
	but DRSGD can not.
	For these two algorithms,
	a communication bottleneck arises from multiple consensus steps,
	namely, multiple rounds of communications per iteration,
	making them difficult to acquire a practical application in large-scale networks.
	As illustrated in \cite{Chen2021accelerating}, 
	one local consensus step requires approximately one-third or one-half
	 the communication overheads of a global averaging operation under a certain practical scenario.
	Consequently, multiple consensus steps will incur intolerably high communication budgets,
	thereby hindering the scalability of decentralized algorithms.
	Furthermore, an extra consensus stepsize, which is less than $1$, 
	is introduced in DRSGD and DRGTA, significantly slowing down the convergence.
	In fact, the above two requirements are imposed to 
	control the consensus error on Stiefel manifolds.
	The nonconvexity of $\Snp$ triggers off an enormous difficulty
	in seeking a consensus on it.

\subsection{Contributions}

	This paper develops an infeasible decentralized approach, 
	called DESTINY, to the problem \eqref{eq:opt-dest},
	which invokes only a single round of communications per iteration
	and gets rid of small consensus stepsizes.
	These two goals are achieved with the help of 
	a novel approximate augmented Lagrangian function.
	It enables us to use unconstrained algorithms to solve the problem \eqref{eq:opt-dest}
	and hence tides us over the obstacle to reaching a consensus on Stiefel manifolds.
	The potential of this function is not limited to the decentralized setting.
	It can serve as a penalty function for optimization problems over the Stiefel manifold,
	which is of independent interest.
	
	Although DESTINY is based on the gradient tracking technique,
	existing convergence results are not applicable.
	We rigorously establish the global convergence to stationary points
	with the worst case complexity under rather mild conditions.
	All the theoretical analysis is applicable to the special case of $d = 1$.
	Therefore, as a by-product, DESTINY boils down to a new algorithm
	in the non-distributed setting.
	
	Preliminary numerical experiments, conducted under the distributed environment,
	validate the effectiveness of our approximate augmented Lagrangian function
	and demonstrate the robustness to penalty parameters.
	In addition, for the first time, we find that BB stepsizes 
	can accelerate the convergence of decentralized algorithms 
	for optimization problems over the Stiefel manifold.
	Finally, extensive experimental results indicate that our algorithm
	requires far fewer rounds of communications than 
	what are required by existing peer methods.
	
\subsection{Notations}

	The Euclidean inner product of two matrices \(Y_1, Y_2\) 
	with the same size is defined as \(\jkh{Y_1, Y_2}=\tr(Y_1\zz Y_2)\),
	where $\tr (B)$ stands for the trace of a square matrix $B$.
	And the notation $\sym (B) = (B + B\zz) / 2$ represents 
	the symmetric part of $B$.
	The Frobenius norm and 2-norm of a given matrix \(C\) 
	is denoted by \(\norm{C}\ff\) and \(\norm{C}_2\), respectively. 
	The $(i, j)$-th entry of a matrix $C$ is represented by $C(i, j)$.
	The notation $\bfone_d$ stands for the $d$-dimensional vector of all ones.
	The Kronecker product is denoted by $\otimes$.
	Given a differentiable function \(g(X) : \Rnp \to \bR\), 
	the gradient of \(g\) with respect to \(X\) is represented by \(\nabla g(X)\).
	Further notation will be introduced wherever it occurs.
	
\subsection{Outline}

	The rest of this paper is organized as follows. 
	In Section \ref{sec:algorithm}, 
	we introduce an approximate augmented Lagrangian function
	and devise a novel decentralized algorithm for \eqref{eq:opt-dest}.
	Moreover, Section \ref{sec:convergence} is dedicated to investigating
	the convergence properties of the proposed algorithm.
	Numerical experiments on a variety of test problems 
	are presented in Section \ref{sec:numerical} 
	to evaluate the performance of the proposed algorithm.
	We draw final conclusions and discuss some possible future developments in the last section.

\section{Algorithm Development}\label{sec:algorithm}

	In this section, some preliminaries are first introduced,
	including first-order stationarity conditions and mixing matrices.
	Then we propose a novel approximate augmented Lagrangian function
	to deal with nonconvex orthogonality constraints.
	From this foundation, we further devise an efficient infeasible decentralized approach
	to optimization problems over the Stiefel manifold.

\subsection{Preliminaries}\label{subsec:preliminaries}

	To facilitate the narrative, we first introduce the stationarity conditions.
	As discussed in \cite{Absil2008,Wang2021multipliers},
	the definition of first-order stationary points of \eqref{eq:opt-stiefel}
	can be stated as follows.
	
	\begin{definition}
		A point $X \in \Snp$ is called a first-order stationary point of \eqref{eq:opt-stiefel}
		if and only if 
		\begin{equation*}
			0 = \mathrm{grad}\, f(X) 
			:= \nabla f(X) - X \sym \dkh{X\zz \nabla f(X)}.
		\end{equation*}
		Here, $\mathrm{grad}\, f(X)$ denotes 
		the Riemannian gradient \cite{Absil2008} of $f$ at $X$. 
	\end{definition}

	In addition, the structure of the communication network 
	is associated with a mixing matrix $W = [W(i, j)] \in \bR^{d \times d}$,
	which conforms to the underlying graph topology.
	The mixing matrix plays an essential role in the development of decentralized algorithms.
	We assume that $W$ satisfies the following properties,
	which are standard in the literature.
	
	\begin{assumption}\label{asp:network}
		The mixing matrix $W$ satisfies the following conditions.
		
		\begin{enumerate}
			
			\item $W$ is symmetric.
			
			\item $W$ is doubly stochastic, namely, $W$ is nonnegative 
			and $W \mathbf{1}_d = W\zz \mathbf{1}_d = \mathbf{1}_d$.
			
			\item $W(i, j) = 0$ if $i \neq j$ and $(i, j) \notin \tE$.
			
		\end{enumerate}
		
	\end{assumption}
	
	Typically, decentralized algorithms rely on the mixing matrix $W$ 
	to diffuse information throughout the network,
	which has a few common choices.
	We refer interested readers to \cite{Xiao2004,Shi2015} for more details.
	According to the Perron-Frobenius Theorem \cite{Pillai2005perron},
	we know that the eigenvalues of $W$ lie in $(-1, 1]$ and
	\begin{equation}\label{eq:lambda}
		\lambda := \norm{W - \bfone_d \bfone_d\zz / d}_2 < 1.
	\end{equation}
	The parameter $\lambda$ measures the connectedness of networks
	and plays a prominent part in the analysis of gradient tracking based methods.

\subsection{Approximate augmented Lagrangian function}\label{subsec:penalty}
	
	As mentioned earlier, it is quite difficult to seek a consensus on the Stiefel manifold,
	which motivates us to devise an infeasible decentralized algorithm 
	for the problem \eqref{eq:opt-dest}.
	The renowned augmented Lagrangian method is a natural choice.
	Therefore, the story behind our algorithm starts from the augmented Lagrangian function
	of optimization problems over the Stiefel manifold,
	which is of the following form.
	\begin{equation*}\label{eq:lagrangian} 
		\cL (X, \Lambda) 
		:= f (X) - \dfrac{1}{2}\jkh{\Lambda, X\zz X - I_p} 
		+ \dfrac{\beta}{4} \norm{X\zz X - I_p}\fs,
	\end{equation*}
	where $\Lambda \in \Rpp$ consists of the corresponding Lagrangian multipliers
	and $\beta > 0$ is the penalty parameter.
	However, the presence of multipliers brings the difficulties 
	in both algorithmic design and theoretical analysis.

	In fact, the Lagrangian multipliers $\Lambda$, as discussed in \cite{Wang2021multipliers}, 
	admit the following closed-form expression 
	at any first-order stationary point $X$ of \eqref{eq:opt-stiefel}.
	\begin{equation}\label{eq:multiplier}
		\Lambda = \sym \dkh{X\zz \nabla f (X)}.
	\end{equation}
	The second term of $\cL (X, \Lambda)$ with the above expression of $\Lambda$
	can be rewritten as the following form.
	\begin{equation*}
		\jkh{\Lambda, X\zz X - I_p} = \jkh{\nabla f (X), X X\zz X - X}.
	\end{equation*}
	Then in light of the Taylor expansion, we have the following approximation.
	\begin{equation}\label{eq:approximate}
		f (X X\zz X) - f (X) \approx \jkh{\nabla f (X), X X\zz X - X}.
	\end{equation}
	Now we replace the second term of $\cL (X, \Lambda)$ 
	with the left hand side of \eqref{eq:approximate},
	and then obtain the following function.
	\begin{equation*}
		h (X) := g (X) + \beta b (X),
	\end{equation*}
	where
	\begin{equation*}
		g (X) := \dfrac{3}{2}\, f (X) - \dfrac{1}{2}\, f (X X\zz X),
		\mbox{~and~}
		b (X) := \dfrac{1}{4} \norm{X\zz X - I_p}\fs.
	\end{equation*}
	We call it \textit{approximate augmented Lagrangian function}.
	The multiplier is updated implicitly by the closed form expression \eqref{eq:multiplier}
	in this function $h (X)$,
	and hence, we can design an infeasible algorithm entirely in the primal formulation.
	
	Now we discuss the gradient of $h (X)$ that can be easily computed as follows.
	\begin{equation*}
		\nabla h (X) = \nabla g (X) + \beta X \dkh{X\zz X - I_p},
	\end{equation*}
	where
	\begin{equation*}
		\nabla g (X) 
		= \dfrac{3}{2} \nabla f (X) - \dfrac{1}{2} \nabla f (X X\zz X) X\zz X 
		- X \sym \dkh{X\zz \nabla f (X X\zz X)}.
	\end{equation*}
	Clearly, the gradient of $f$ should be evaluated twice
	at different points $X$ and $X X\zz X$ in $\nabla g (X)$,
	which gives rise to additional computational costs.
	In order to alleviate this difficulty,
	we use the following direction to approximate $\nabla g (X)$.
	\begin{equation*}
		G (X) = \dfrac{3}{2} \nabla f (X X\zz X) - \dfrac{1}{2} \nabla f (X X\zz X) X\zz X
		- X \sym \dkh{X\zz \nabla f (X X\zz X)}.
	\end{equation*}
	For convenience, we denote
	\begin{equation*}
		H (X) = G (X) + \beta X \dkh{X\zz X - I_p}.
	\end{equation*}
	As an approximation of $\nabla h (X)$,
	$H (X)$ possesses the following desirable properties.
	
	\begin{lemma} \label{le:expen}
		Let $\cR := \{X \in \Rnp \mid \|X\zz X - I_p\|\ff \leq 1 / 6\}$
		be a bounded region 
		and $C_0 := \sup_{X \in \cR} \norm{\nabla f (X X\zz X)}\ff$ 
		be a positive constant.
		Then if $\beta \geq (6 + 21 C_0) / 5$, we have
		\begin{equation*}
			\norm{H (X)}\fs \geq \norm{G (X)}\fs 
			+ \beta \norm{X\zz X - I_p}\fs,
		\end{equation*}
		for any $X \in \cR$.
	\end{lemma}

	\begin{proof}
		Suppose $\sigma_{\min}$ is the smallest singular value of $X$.
		Then for any $X \in \cR$, we have $\norm{X}_2^2 \leq 7 / 6$ 
		and $\sigma_{\min}^2 \geq 5 / 6$, and hence,
		\begin{equation*}
			\norm{X \dkh{X\zz X - I_p}}\fs 
			\geq \sigma_{\min}^2 \norm{X\zz X - I_p}\fs
			\geq \dfrac{5}{6} \norm{X\zz X - I_p}\fs.
		\end{equation*}
		Moreover, simple algebraic manipulations give us
		\begin{equation*}
			\begin{aligned}
				& \jkh{G (X), X \dkh{X\zz X - I_p}} \\
				& = \jkh{\nabla f (X X\zz X) \dkh{\dfrac{3}{2}I_p - \dfrac{1}{2} X\zz X}, 
					X \dkh{X\zz X - I_p}}
				- \jkh{X \sym \dkh{X\zz \nabla f (X X\zz X)}, X \dkh{X\zz X - I_p}} \\
				& = \jkh{\sym \dkh{X\zz \nabla f (X X\zz X)}, 
					\dkh{X\zz X - I_p} \dkh{\dfrac{3}{2}I_p - \dfrac{1}{2} X\zz X} 
					- X\zz X\dkh{X\zz X - I_p}} \\
				& = - \dfrac{3}{2} \jkh{\sym \dkh{X\zz \nabla f (X X\zz X)}, 
					\dkh{X\zz X - I_p}^2},
			\end{aligned}
		\end{equation*}
		which yields that
		\begin{equation*}
			\begin{aligned}
				\abs{\jkh{G (X), X \dkh{X\zz X - I_p}}}
				& \leq \dfrac{3}{2} \norm{\sym \dkh{X\zz \nabla f (X X\zz X)}}\ff 
				\norm{\dkh{X\zz X - I_p}^2}\ff \\
				& \leq \dfrac{7}{4} C_0 \norm{X\zz X - I_p}\fs.
			\end{aligned}
		\end{equation*}
		Now it can be readily verifies that
		\begin{equation*}
			\begin{aligned}
				\norm{H (X)}\fs 
				& = \norm{G (X)}\fs + 2 \beta \jkh{G (X), X \dkh{X\zz X - I_p}}
				+ \beta^2 \norm{X \dkh{X\zz X - I_p}}\fs \\
				& \geq \norm{G (X)}\fs 
				+ \dfrac{1}{6} \beta \dkh{5 \beta - 21 C_0}\norm{X\zz X - I_p}\fs,
			\end{aligned}
		\end{equation*}
		which together with $\beta \geq (6 + 21 C_0) / 5$ completes the proof.
	\end{proof}
	
	Lemma \ref{le:expen} reveals that the orthogonality violation of $X$
	is controlled by the squared norm of $H (X)$ 
	as long as $X \in \cR$ and $\beta$ is sufficiently large.
	And it is worthy of mentioning that $G (X) = \mathrm{grad}\, f(X)$ when $X \in \Snp$.
	In fact, $h (X)$ can play the part of a penalty function 
	for optimization problems over the Stiefel manifold,
	which is of interest for a future investigation.
	
\subsection{Algorithm description}	

	Inspired by Lemma \ref{le:expen},
	we can devise a decentralized algorithm such that 
	the iterates are restricted in the bounded region $\cR$
	and $H (X)$ asymptotically vanishes.

	Under the decentralized setting, 
	the computation of $H (X)$ can be distributed into $d$ agents as follows.
	\begin{equation*}
		H (X) = \dfrac{1}{d} \sumiid H_i (X),
	\end{equation*}
	where $H_i (X) := G_i (X) + \beta X (X\zz X - I_p)$, and 
	\begin{equation*}
		G_i (X) := \nabla f_i (X X\zz X) \dkh{3 I_p - X\zz X} / 2 
		- X \sym \dkh{X\zz \nabla f_i (X X\zz X)}.
	\end{equation*}
	Clearly, only local gradient $\nabla f_i (X X\zz X)$ 
	is involved in the expression of $H_i (X)$,
	which signifies that the evaluation of $H_i (X)$
	can be accomplished by the $i$-th agent individually. 
	Consequently, $H_i (X)$ can act as a local descent direction.
	We can employ gradient tracking techniques to estimate
	the average of these local descent directions.
	The resulting algorithm is summarized in Algorithm \ref{alg:DESTINY},
	named \textit{\underline{de}centralized \underline{Sti}efel algorithm
	by an approximate augmented Lagra\underline{n}gian penalt\underline{y} function} 
	and abbreviated to DESTINY.
	
	\begin{algorithm}[ht!]
		\caption{DESTINY.} 
		\label{alg:DESTINY}
		
		\KwIn{initial guess $X_{\mathrm{initial}} \in \Snp$, stepsize $\eta > 0$, 
		and penalty parameter $\beta > 0$.}
		
		\For{all $\iid$ in parallel}
		{
			
			Set $k := 0$.
			
			Initialize $\Xik := X_{\mathrm{initial}}$ and $\Dik := H_i (\Xik)$. 
		
			\While{``not converged"}
			{
				
				Update $\Xikn := \sumjjd W(i, j) \dkh{\Xjk - \eta \Djk}$.
				
				Update $\Dikn := \sumjjd W(i, j) \Djk + H_i (\Xikn) - H_i (\Xik)$.
					
				Set $k := k + 1$.
				
			}
		
		}
		
		\KwOut{$\{\Xik\}$.}

	\end{algorithm}
	
	Our algorithm consists of two procedures, namely,
	updating local variables and tracking global descent directions.
	Unlike DRGTA \cite{Chen2021decentralized}, in each procedure,
	our algorithm only performs one consensus step per iteration.
	These two operations can be realized
	in a single round of communications to reduce the latency.
	In addition, our algorithm does not involve an extra consensus stepsize.
	
	For the sake of convenience, we now define the following averages of local variables.
	\begin{multicols}{3}
		
		\begin{itemize}
			
			\item $\barXk = \dfrac{1}{d} \sumiid \Xik$.
			
			\item $\barDk =  \dfrac{1}{d} \sumiid \Dik$.
			
			\item $\barHk =  \dfrac{1}{d} \sumiid H_i (\Xik)$.
			
		\end{itemize}
	
	\end{multicols}
	\noindent Then it is not difficult to check that, for any $k \in \bN$,
	the following relationships hold.
	\begin{equation}\label{eq:avg-iter}
		\barXkn = \barXk - \eta \barDk, 
		\mbox{~and~}\barDk = \barHk.
	\end{equation}
	Moreover, we denote $J = \bfone_d \bfone_d\zz / d \in \bR^{d \times d}$,
	$\bfJ = J \otimes I_n \in \bR^{dn \times dn}$, 
	and $\bfW = W \otimes I_n \in \bR^{dn \times dn}$.
	The following stacked notations are also used in the sequel.
	\begin{itemize}
		
		\item $\bfXk = [X_{1, k}\zz, \dotsc, X_{d, k}\zz]\zz \in \bR^{dn \times p}$,  
		$\avXk = \dkh{\bfone_d \otimes I_n} \barXk = \bfJ \bfXk \in \bR^{dn \times p}$.
		
		\item $\bfDk = [D_{1, k}\zz, \dotsc, D_{d, k}\zz]\zz \in \bR^{dn \times p}$,
		$\avDk = \dkh{\bfone_d \otimes I_n} \barDk = \bfJ \bfDk \in \bR^{dn \times p}$.
		
		\item $\bfHk = [H_1 (X_{1, k})\zz, \dotsc, H_d (X_{d, k})\zz]\zz \in \bR^{dn \times p}$,
		$\avHk = \dkh{\bfone_d \otimes I_n} \barHk = \bfJ \bfHk \in \bR^{dn \times p}$.
		
	\end{itemize}
	By virtue of the above notations, we have 
	$\dkh{\bfW - \bfJ} \avXk = \dkh{\bfW - \bfJ} \avDk = 0$.
	And the main iteration loop of Algorithm \ref{alg:DESTINY} 
	can be summarized as the following compact form.
	\begin{equation*}
		\left\{
		\begin{aligned}
			\bfXkn & = \bfW \dkh{\bfXk - \eta \bfDk}, \\
			\bfDkn & = \bfW \bfDk + \bfHkn - \bfHk.
		\end{aligned}
		\right.
	\end{equation*}

\section{Convergence Analysis}\label{sec:convergence}

	Existing convergence guarantees of gradient tracking based algorithms, 
	such as \cite{Daneshmand2020}, 
	are constructed for globally Lipschitz smooth and coercive functions,
	which are restrictive for \eqref{eq:opt-dest} since $\Snp$ is compact.
	In this section, the global convergence of Algorithm~\ref{alg:DESTINY} 
	is rigorously established under rather mild conditions.
	The objective function is only assumed to be locally Lipschitz smooth 
	(see Assumption \ref{asp:smooth}).
	The worst case complexity is also provided.
	
	All the special constants to be used in this section are listed below.
	We divide these constants into two categories.
	Recall that $\cR$ is a bounded set defined in Lemma \ref{le:expen}
	and $\lambda$ is a constant defined in \eqref{eq:lambda}.
	And we define $\cB := \{X \in \Rnp \mid \norm{X}\ff \leq \sqrt{7dp / 6} + \sqrt{d}\}$.
	\begin{itemize}
		
		\item Category I:
		\begin{equation} \label{eq:constants-1}
			\begin{aligned}
				& M_g = \sup\, \{ \norm{G_i (X)}\ff \mid X \in \cB, \iid \}; \quad
				M_b = \sup\, \{ \norm{\nabla b (X)}\ff \mid X \in \cB\}; \\
				& L_g = \sup\, \hkh{ 
				\dfrac{\norm{G_i (X) - G_i (Y)}\ff}{\norm{X - Y}\ff}
				\bigg| X \neq Y, X \in \cB, Y \in \cB, \iid}; \\
				& L_b = \sup\, \hkh{ 
				\dfrac{\norm{\nabla b (X) - \nabla b (Y)}\ff}{\norm{X - Y}\ff}
				\bigg| X \neq Y, X \in \cB, Y \in \cB}; \quad
				C_1 = \dfrac{\lambda^2 \dkh{1 + \lambda^2}}{1 - \lambda^2}; \\
				&  L_f = \sup\, \hkh{ 
				\dfrac{\norm{\nabla f (X) - \nabla f (X X\zz X)}\ff}{\norm{X\dkh{X\zz X - I_p}}\ff}
				\bigg| X \neq X X\zz X , X \in \cB}; \quad
				\gamma = \dfrac{1 - \lambda^2}{2 \lambda^2}.
			\end{aligned}
		\end{equation}
	
		\item Category II:
		\begin{equation} \label{eq:constants-2}
			\begin{aligned}
				& C_D = \norm{\bar{\mathbf{D}}_{0} - \mathbf{D}_{0}}\ff 
				+ 3 \sqrt{d} / \dkh{1 - \lambda}
				+ \sqrt{d} \dkh{M_g + \beta M_b}; \\
				& C_2 = \dfrac{12 \dkh{1 + \lambda^2} \dkh{L_g + \beta L_b}^2}{1 - \lambda^2}; \quad
				\rho = \dfrac{3}{16 C_2} \dkh{1 - \lambda^2}.
			\end{aligned}
		\end{equation}
	
	\end{itemize}

	The first category of constants defined in \eqref{eq:constants-1}
	is independent of $\beta$,
	while the second category in \eqref{eq:constants-2} is not.
	In order to establish the global convergence of Algorithm \ref{alg:DESTINY},
	we need to impose 
	several mild conditions on $\beta$  and $\eta$.
	To facilitate the narrative, we first state all these conditions here.
	
	\begin{assumption}\label{asp:parameters}
		We make the following assumptions about $\beta$ and $\eta$. 
		\begin{enumerate}
			
			\item The penalty parameter $\beta$ satisfies 
			\begin{equation*}
				\beta > \max\hkh{
					\dfrac{6 + 21 C_0}{5}, \;
					\dfrac{72 (4 + 3 M_g)}{5}, \;
					\dfrac{1}{L_b}, \;
					22 L_f^2
				}.
			\end{equation*}
			
			\item The stepsize $\eta$ satisfies
			\begin{equation*}
				0 < \eta < \min \hkh{
					\bar{\eta}_1,
					\bar{\eta}_2,
					\bar{\eta}_3,
					\bar{\eta}_4,
					\bar{\eta}_5,
					\bar{\eta}_6,
					\bar{\eta}_7,
					\bar{\eta}_8,
					\bar{\eta}_9,
					\bar{\eta}_{10},
					\bar{\eta}_{11},
					\bar{\eta}_{12}
				},
			\end{equation*}
			where
			\begin{equation*}
				\begin{aligned}
					& \bar{\eta}_1 = \dfrac{216}{49 \beta^2}, \quad 
					\bar{\eta}_2 = \dfrac{12}{7 \beta}, \quad
					\bar{\eta}_3 = \dfrac{1}{18(4 + 3M_g)}, \quad
					\bar{\eta}_4 = \dfrac{5 \beta / 72 - 4 - 3 M_g}{(1 + M_g)^2},  \\
					& \bar{\eta}_5 
					= \dfrac{\sqrt{d} (1 - \lambda)}{(L_g + \beta L_b) \lambda C_D}, \quad
					\bar{\eta}_6 = \dfrac{\sqrt{d}}{(L_g + \beta L_b) C_D},  \quad
					\bar{\eta}_7 = \dfrac{1}{8 (L_g + \beta L_b)}, \\
					& \bar{\eta}_8 = \dfrac{1}{4 \rho d C_2}, \quad
					\bar{\eta}_9 = \dfrac{1 - \lambda^2}{2\lambda^2 (L_g + \beta L_b)}
					\sqrt{\dfrac{1}{6 (1 + \lambda^2)}}, \\
					& \bar{\eta}_{10} = \sqrt{ \dfrac{d(1 - \lambda^2)}{24 (L_g + \beta L_b)^3}}, \quad
					\bar{\eta}_{11} = \dfrac{d (1 - \lambda^2)}{12 (L_g + \beta L_b)^2}, \quad
					\bar{\eta}_{12} = \sqrt{\dfrac{3 \rho (1 - \lambda^2)}{16 C_1}}.
				\end{aligned}
			\end{equation*}
		
		\end{enumerate}

	\end{assumption}

	The conditions in Assumption \ref{asp:parameters} 
	are introduced for theoretical analysis. 
	The parameters $\beta$  and $\eta$ satisfying these conditions 
	are usually restrictive in practical use.

\subsection{Boundedness of iterates}

	The purpose of this subsection is to show that 
	the sequence $\{\dkh{\bfXk, \bfDk}\}$ generated by Algorithm \ref{alg:DESTINY} is bounded 
	and $\barXk$ is restricted in the bounded region $\cR$.
	We first prove the following technical lemma.

	\begin{lemma} \label{le:omega}
		Suppose the conditions in Assumptions \ref{asp:smooth}, \ref{asp:network}, 
		and \ref{asp:parameters} hold,
		and $\barXkn$ is generated by \eqref{eq:avg-iter}
		with $\barXk \in \cR$, $\norm{\bfXk}\ff \leq \sqrt{7dp / 6} + \sqrt{d}$,
		and $ \norm{\avXk - \bfXk}\ff \leq \sqrt{d} / (L_g + \beta L_b)$.
		Then we have $\barXkn \in \cR$.
	\end{lemma}

	\begin{proof}
		Since $\norm{\bfXk}\ff \leq \sqrt{7dp / 6} + \sqrt{d}$,
		we have $\Xik \in \cB$.
		Then it can be readily verified that
		\begin{equation*}
			\begin{aligned}
				\norm{ H (\barXk) - \barHk}\ff
				& \leq \dfrac{1}{d} \sumiid \norm{H_i (\barXk) - H_i (\Xik)}\ff
				\leq \dfrac{L_g + \beta L_b}{d} \sumiid \norm{\barXk - \Xik}\ff \\
				& \leq \dfrac{L_g + \beta L_b}{\sqrt{d}} \norm{\avXk - \bfXk}\ff
				\leq 1.
			\end{aligned}
		\end{equation*}
		Moreover, it follows from \eqref{eq:avg-iter} that 
		\begin{equation*}
			\begin{aligned}
				\barXkn = \barXk - \eta H (\barXk) + \eta \dkh{H (\barXk) - \barHk}
				= \barXk - \eta \beta \barXk \dkh{\barXk\zz \barXk - I_p} + \eta Y_k,
			\end{aligned}
		\end{equation*}
		where $Y_k := H (\barXk) - \barHk - G (\barXk)$ and $Y_k$ satisfies
		\begin{equation*}
			\norm{Y_k}\ff \leq \norm{H (\barXk) - \barHk}\ff 
			+ \norm{G (\barXk)}\ff
			\leq 1 + M_g.
		\end{equation*}
		Then by straightforward calculations, we can obtain that
		\begin{equation*}
			\begin{aligned}
				\barXkn\zz \barXkn - I_p
				= {} & \dkh{\barXk - \eta \beta \barXk \dkh{\barXk\zz \barXk - I_p} + \eta Y_k}\zz
				\dkh{\barXk - \eta \beta \barXk \dkh{\barXk\zz \barXk - I_p} + \eta Y_k} - I_p \\
				= {} & \barXk\zz \barXk - I_p - 2 \eta \beta \barXk\zz \barXk \dkh{\barXk\zz \barXk - I_p} 
				+ \eta \barXk\zz Y_k + \eta^2 \beta^2 \barXk\zz \barXk \dkh{\barXk\zz \barXk - I_p}^2 \\
				& - \eta^2 \beta \dkh{\barXk\zz \barXk - I_p} \barXk\zz Y_k + \eta Y_k\zz \barXk
				- \eta^2 \beta Y_k\zz \barXk \dkh{\barXk\zz \barXk - I_p} + \eta^2 Y_k\zz Y_k \\
				= {} & \dkh{I_n - \eta \beta \barXk\zz \barXk}^2 \dkh{\barXk\zz \barXk - I_p}
				- \eta^2 \beta^2 \barXk\zz \barXk \dkh{\barXk\zz \barXk - I_p} 
				+ \eta \barXk\zz Y_k + \eta Y_k\zz \barXk \\
				&  - \eta^2 \beta \dkh{\barXk\zz \barXk - I_p} \barXk\zz Y_k
				- \eta^2 \beta Y_k\zz \barXk \dkh{\barXk\zz \barXk - I_p} + \eta^2 Y_k\zz Y_k.
			\end{aligned}
		\end{equation*}
		This further implies that
		\begin{equation*}
			\begin{aligned}
				\norm{\barXkn\zz \barXkn - I_p}\ff 
				\leq {} & \norm{\dkh{I_n - \eta \beta \barXk\zz \barXk}^2 \dkh{\barXk\zz \barXk - I_p}}\ff 
				+ \eta^2 \beta^2 \norm{ \barXk\zz \barXk \dkh{\barXk\zz \barXk - I_p} }\ff \\
				& + 2 \eta \norm{\barXk\zz Y_k}\ff 
				+ 2 \eta^2 \beta \norm{\dkh{\barXk\zz \barXk - I_p} \barXk\zz Y_k}\ff 
				+ \eta^2 \norm{Y_k\zz Y_k}\ff \\
				\leq {} & \dkh{1 - \dfrac{5}{6} \eta \beta}^2 \norm{\barXk\zz \barXk - I_p}\ff 
				+ \dfrac{49}{216} \eta^2 \beta^2 + \dfrac{7}{3} (1 + M_g) \eta \\
				& + \dfrac{7}{18} (1 + M_g) \eta^2 \beta + (1 + M_g)^2 \eta^2 \\
				\leq {} & \dkh{1 - \dfrac{5}{6} \eta \beta}^2 \norm{\barXk\zz \barXk - I_p}\ff
				+ \eta \dkh{4 + 3 M_g} + \eta^2 \dkh{1 + M_g}^2, 
			\end{aligned}
		\end{equation*}
		where the last equality follows from 
		$\eta \leq \min\{\bar{\eta}_1, \bar{\eta}_2\}$.
		Now we consider the above relationship in the following two cases.
		
		\paragraph{Case I: $\norm{\barXk\zz \barXk - I_p}\ff \leq 1/ 12$.} 
		Since $\eta \leq \bar{\eta}_3$, we have
		\begin{equation*}
			\norm{\barXkn\zz \barXkn - I_p}\ff 
			\leq \norm{\barXk\zz \barXk - I_p}\ff + \dfrac{1}{12} = \dfrac{1}{6}.
		\end{equation*}
		
		\paragraph{Case II: $\norm{\barXk\zz \barXk - I_p}\ff > 1/ 12$.}
		It can be readily verified that
		\begin{equation*}
			\begin{aligned}
				\norm{\barXkn\zz \barXkn - I_p}\ff - \norm{\barXk\zz \barXk - I_p}\ff 
				\leq {} & \dkh{\dkh{1 - \dfrac{5}{6} \eta \beta}^2 - 1}\norm{\barXk\zz \barXk - I_p}\ff \\
				& + \eta \dkh{4 + 3 M_g} + \eta^2 \dkh{1 + M_g}^2 \\
				\leq {} & - \dfrac{5}{72} \eta \beta  + \eta \dkh{4 + 3 M_g} 
				+ \eta^2 \dkh{1 + M_g}^2,
			\end{aligned}
		\end{equation*}
		which together with $\beta > 72 (4 + 3 M_g) / 5$
		and $\eta \leq \bar{\eta}_4$ yields that
		\begin{equation*}
			\norm{\barXkn\zz \barXkn - I_p}\ff - \norm{\barXk\zz \barXk - I_p}\ff \leq 0.
		\end{equation*}
		Hence, we arrive at 
		$\norm{\barXkn\zz \barXkn - I_p}\ff \leq \norm{\barXk\zz \barXk - I_p}\ff \leq 1 / 6$.
		Combing the above two cases, we complete the proof.
	\end{proof}

	Based on Lemma \ref{le:omega},
	we can prove the main results of this subsection.

	\begin{proposition}\label{prop:bound}
		Suppose the conditions in Assumptions \ref{asp:smooth}, \ref{asp:network}, 
		and \ref{asp:parameters} hold.
		Let $\{\dkh{\bfXk, \bfDk}\}$ be the iterate sequence 
		generated by Algorithm \ref{alg:DESTINY} with $X_{\mathrm{initial}} \in \cR$.
		Then for any $k \in \bN$, it holds that 
		\begin{equation}\label{eq:bound-x}
			\barXk \in \cR,~ 
			\norm{\avXk - \bfXk}\ff \leq \dfrac{\sqrt{d}}{L_g + \beta L_b},~
			\norm{\bfXk}\ff \leq \sqrt{\dfrac{7dp}{6}} + \sqrt{d}, 
			\mbox{~and~} \norm{\bfDk}\ff \leq C_D.
		\end{equation}
	\end{proposition}

	\begin{proof}
		We use mathematical induction to prove this proposition.
		The argument \eqref{eq:bound-x} directly holds 
		at $\{\dkh{\mathbf{X}_0, \mathbf{D}_0}\}$ resulting from the initialization.
		Now, we assume that this argument holds at $\{\dkh{\bfXk, \bfDk}\}$,
		and investigate the situation at $\{\dkh{\bfXkn, \bfDkn}\}$.
		
		Our first purpose is to show that $\norm{\avXkn - \bfXkn}\ff \leq \sqrt{d} / (L_g + \beta L_b)$.
		By straightforward calculations, we can attain that
		\begin{equation*}
			\begin{aligned}
				\norm{\avXkn - \bfXkn}\ff
				& = \norm{\dkh{\bfW - \bfJ} \dkh{\avXk - \bfXk} 
					+ \eta \dkh{\bfW - \bfJ} \bfDk}\ff \\
				& \leq \norm{\dkh{\bfW - \bfJ} \dkh{\avXk - \bfXk}}\ff
				+ \eta \norm{\dkh{\bfW - \bfJ} \bfDk}\ff \\
				& \leq \lambda \norm{\avXk - \bfXk}\ff
				+ \eta \lambda \norm{\bfDk}\ff \\
				& \leq \lambda \norm{\avXk - \bfXk}\ff
				+ \eta \lambda C_D,
			\end{aligned}
		\end{equation*}
		which together with $\eta \leq \bar{\eta}_5$ implies that
		\begin{equation*}
			\norm{\avXkn - \bfXkn}\ff 
			\leq \dfrac{\sqrt{d}  \lambda}{L_g + \beta L_b}
			+ \dfrac{\sqrt{d} \dkh{1 - \lambda}}{L_g + \beta L_b}
			\leq \dfrac{\sqrt{d}}{L_g + \beta L_b}.
		\end{equation*}
	
		Then we aim to prove that $\barXkn \in \cR$ 
		and $\norm{\bfXkn}\ff \leq \sqrt{7dp/6} + \sqrt{d}$.
		According to Lemma \ref{le:omega}, we have $\barXkn \in \cR$.
		And it follows that
		\begin{equation*}
			\norm{\bfXkn}\ff 
			\leq \norm{\avXkn}\ff + \norm{\avXkn - \bfXkn}\ff
			\leq \sqrt{\dfrac{7dp}{6}} + \dfrac{\sqrt{d}}{L_g + \beta L_b}
			\leq \sqrt{\dfrac{7dp}{6}} + \sqrt{d},
		\end{equation*}
		as a result of the condition $L_g + \beta L_b \geq \beta L_b \geq 1$.
		
		In order to finish the proof, we still have to show that $\norm{\bfDkn}\ff \leq C_D$.
		In fact, we have 
		\begin{equation*}
			\norm{H_i (\Xikn) - H_i (\Xik)}\ff 
			\leq (L_g + \beta L_b) \norm{\Xikn - \Xik}\ff, \quad \iid,
		\end{equation*}
		which implies that
		\begin{equation*}
			\norm{\bfHkn - \bfHk}\ff \leq (L_g + \beta L_b) \norm{\bfXkn - \bfXk}\ff.
		\end{equation*}
		Moreover, it can be readily verified that
		\begin{equation*}
			\begin{aligned}
				\norm{\bfXkn - \bfXk}\ff 
				& = \norm{\bfW \dkh{\bfXk - \eta \bfDk} - \bfXk}\ff
				= \norm{\dkh{I_{dn} - \bfW} \dkh{\avXk - \bfXk}- \eta \bfW \bfDk}\ff \\
				& \leq 2 \norm{\avXk - \bfXk}\ff + \eta \norm{\bfDk}\ff
				\leq \dfrac{2 \sqrt{d}}{L_g + \beta L_b} + \eta C_D
				\leq \dfrac{3 \sqrt{d}}{L_g + \beta L_b},
 			\end{aligned}
		\end{equation*}
		where the last inequality follows from $\eta \leq \bar{\eta}_6$.
		Combing the above two relationships, we can obtain that
		\begin{equation*}
			\begin{aligned}
				\norm{\avDkn - \bfDkn}\ff
				& = \norm{ \dkh{\bfW - \bfJ} \dkh{\avDk - \bfDk} 
					- \dkh{I_{dn} - \bfJ} \dkh{\bfHkn - \bfHk} }\ff \\
				& \leq \norm{\dkh{\bfW - \bfJ} \dkh{\avDk - \bfDk}}\ff
				+ \norm{\dkh{I_{dn} - \bfJ} \dkh{\bfHkn - \bfHk}}\ff \\
				& \leq \lambda \norm{\avDk - \bfDk}\ff
				+ \norm{\bfHkn - \bfHk}\ff \\
				& \leq \lambda \norm{\avDk - \bfDk}\ff + 3 \sqrt{d},
			\end{aligned}
		\end{equation*}
		which is followed by
		\begin{equation*}
			\norm{\avDkn - \bfDkn}\ff 
			\leq \norm{\bar{\mathbf{D}}_{0} - \mathbf{D}_{0}}\ff 
			+ \dfrac{3 \sqrt{d}}{1 - \lambda}.
		\end{equation*}
		Therefore, we can obtain that
		\begin{equation*}
			\begin{aligned}
				\norm{\bfDkn}\ff 
				& \leq \norm{\avDkn - \bfDkn}\ff + \norm{\avDkn}\ff
				= \norm{\avDkn - \bfDkn}\ff + \norm{\avHkn}\ff \\
				& \leq \norm{\bar{\mathbf{D}}_{0} - \mathbf{D}_{0}}\ff 
				+ \dfrac{3 \sqrt{d}}{1 - \lambda}
				+ \sqrt{d} \dkh{M_g + \beta M_b} 
				= C_D. 
			\end{aligned}
		\end{equation*}
		The proof is completed.
	\end{proof}

\subsection{Sufficient descent of the merit function}
	
	In this subsection, we aim to prove the sufficient descent property of
	the merit function (to be defined later). 
	Towards this end, we first build the upper bound of 
	consensus error $\|\avXk - \bfXk\|\fs$ 
	and gradient tracking error $\|\avDk - \bfDk\|\fs$.

	\begin{lemma}\label{le:consensus-error}
		Suppose $\{\dkh{\bfXk, \bfDk}\}$ is the iterate sequence generated by Algorithm \ref{alg:DESTINY}.
		Then for any $k \in \bN$, it holds that
		\begin{equation*}
			\norm{\avXkn - \bfXkn}\fs
			\leq \dfrac{1 + \lambda^2}{2} \norm{\avXk - \bfXk}\fs
			+ \eta^2 C_1\norm{\avDk - \bfDk}\fs,
		\end{equation*}
		where $C_1 > 0$ is a constant defined in \eqref{eq:constants-1}.
	\end{lemma}
	
	\begin{proof}
		By straightforward calculations, we can attain that
		\begin{equation*}
			\begin{aligned}
				\norm{\avXkn - \bfXkn}\fs
				= & \norm{ \dkh{\bfW - \bfJ} \dkh{\avXk - \bfXk} 
					- \eta \dkh{\bfW - \bfJ} \dkh{\avDk - \bfDk} }\fs \\
				\leq & \dkh{1 + \gamma} \norm{\dkh{\bfW - \bfJ} \dkh{\avXk - \bfXk}}\fs 
				+ \eta^2 \dkh{1 + 1 / \gamma} \norm{\dkh{\bfW - \bfJ} \dkh{\avDk - \bfDk}}\fs \\
				\leq & \lambda^2 \dkh{1 + \gamma} \norm{\avXk - \bfXk}\fs
				+ \eta^2 \lambda^2 \dkh{1 + 1 / \gamma} \norm{\avDk - \bfDk}\fs,
			\end{aligned}
		\end{equation*}
		where $\gamma > 0$ is a constant defined in \eqref{eq:constants-1}.
		This completes the proof 
		since $\lambda^2 \dkh{1 + \gamma} =  (1 + \lambda^2) / 2$
		and $\lambda^2 \dkh{1 + 1 / \gamma} = C_1$.
	\end{proof}

	\begin{lemma}\label{le:dk}
		Suppose the conditions in Assumptions \ref{asp:smooth}, \ref{asp:network}, 
		and \ref{asp:parameters} hold.
		Let $\{\dkh{\bfXk, \bfDk}\}$ be the iterate sequence 
		generated by Algorithm \ref{alg:DESTINY} with $X_{\mathrm{initial}} \in \cR$.
		Then for any $k \in \bN$, it holds that
		\begin{equation*}
			\norm{\avDk}\fs 
			\leq 2 \dkh{L_g + \beta L_b}^2 \norm{\avXk - \bfXk}\fs
			+ 2d \norm{H (\barXk)}\fs.
		\end{equation*}
	\end{lemma}

	\begin{proof}
		At first, it is straightforward to verify that
		\begin{equation}\label{eq:diff-hd}
			\begin{aligned}
				\norm{H (\barXk) - \barDk}\ff
				& = \norm{H (\barXk) - \barHk}\ff
				\leq \dfrac{1}{d} \sumiid \norm{H_i (\barXk) - H_i (\Xik)}\ff \\
				& \leq \dfrac{L_g + \beta L_b}{d} \sumiid \norm{\barXk - \Xik}\ff
				\leq \dfrac{L_g + \beta L_b}{\sqrt{d}} \norm{\avXk - \bfXk}\ff,
			\end{aligned}
		\end{equation}
		which further implies that
		\begin{equation}\label{eq:dk}
			\begin{aligned}
				\norm{\barDk}\fs 
				& \leq 2 \norm{H (\barXk) - \barDk}\fs + 2 \norm{H (\barXk)}\fs \\
				& \leq \dfrac{2}{d} \dkh{L_g + \beta L_b}^2 \norm{\avXk - \bfXk}\fs
				+ 2 \norm{H (\barXk)}\fs.
			\end{aligned}
		\end{equation}
		We complete the proof since $\|\avDk\|\fs = d \|\barDk\|\fs$.
	\end{proof}
	
	\begin{lemma}\label{le:tracking-error}
		Suppose the conditions in Assumptions \ref{asp:smooth}, \ref{asp:network}, 
		and \ref{asp:parameters} hold.
		Let $\{\dkh{\bfXk, \bfDk}\}$ be the iterate sequence 
		generated by Algorithm \ref{alg:DESTINY} with $X_{\mathrm{initial}} \in \cR$.
		Then for any $k \in \bN$, it holds that
		\begin{equation*}
			\norm{\avDkn - \bfDkn}\fs
			\leq {} \dfrac{5 + 3 \lambda^2}{8} \norm{\avDk - \bfDk}\fs
			+ C_2 \norm{\avXk - \bfXk}\fs
			+ \dfrac{\eta}{8 \rho} \norm{H (\barXk)}\fs,
		\end{equation*}
		where $C_2 > 0$ and $\rho > 0$ are two constants 
		defined in \eqref{eq:constants-2}.
	\end{lemma}
	
	\begin{proof}
		By straightforward calculations, we can attain that
		\begin{equation*}
			\begin{aligned}
				\norm{\avDkn - \bfDkn}\fs
				& = \norm{ \dkh{\bfW - \bfJ} \dkh{\avDk - \bfDk} 
					- \dkh{I_{dn} - \bfJ} \dkh{\bfHkn - \bfHk} }\fs \\
				& \leq \dkh{1 + \gamma} \norm{\dkh{\bfW - \bfJ} \dkh{\avDk - \bfDk}}\fs 
				+ \dkh{1 + 1 / \gamma} \norm{\dkh{I_{dn} - \bfJ} \dkh{\bfHkn - \bfHk}}\fs \\
				& \leq \lambda^2 \dkh{1 + \gamma} \norm{\avDk - \bfDk}\fs
				+ \dkh{1 + 1 / \gamma} \norm{\bfHkn - \bfHk}\fs,
			\end{aligned}
		\end{equation*}
		where $\gamma > 0$ is a constant defined in \eqref{eq:constants-1}.
		According to Proposition \ref{prop:bound}, 
		the inequality $\norm{\bfXk}\ff \leq \sqrt{7dp / 6} + \sqrt{d}$ 
		holds for all $k \in \bN$.
		Hence, it follows that
		\begin{equation*}
			\norm{\bfHkn - \bfHk}\ff \leq (L_g + \beta L_b) \norm{\bfXkn - \bfXk}\ff.
		\end{equation*}
		Moreover, it can be readily verified that
		\begin{equation*}
			\begin{aligned}
				\norm{\bfXkn - \bfXk}\fs
				& = \norm{\bfW \dkh{\bfXk - \eta \bfDk} - \bfXk}\fs \\
				& = \norm{\dkh{I_{dn} - \bfW} \dkh{\avXk - \bfXk} 
				+ \eta \dkh{\bfW - \bfJ} \dkh{\avDk - \bfDk} - \eta \bfW \avDk}\fs \\
				& \leq 6 \norm{\avXk - \bfXk}\fs + 3 \eta^2 \lambda^2 \norm{\avDk - \bfDk}\fs
				+ 3 \eta^2 \norm{\avDk}\fs.
			\end{aligned}
		\end{equation*}
		According to Lemma \ref{le:dk}, we have
		\begin{equation*}
			\begin{aligned}
			\norm{\bfXkn - \bfXk}\fs 
			& \leq 6 \dkh{1 + \eta^2 \dkh{L_g + \beta L_b}^2 }\norm{\avXk - \bfXk}\fs
			 + 3 \eta^2 \lambda^2 \norm{\avDk - \bfDk}\fs
			+ 6 d \eta^2 \norm{H (\barXk)}\fs \\
			&\leq 12 \norm{\avXk - \bfXk}\fs
			+ 3 \eta^2 \lambda^2 \norm{\avDk - \bfDk}\fs
			+ \dfrac{3 \eta}{2 \rho C_2} \norm{H (\barXk)}\fs,
			\end{aligned}
		\end{equation*}
		where the last inequality follows from the condition
		$\eta \leq \min\{\bar{\eta}_7, \bar{\eta}_8\}$.
		Combing the above relationships, we can obtain that
		\begin{equation*}
			\begin{aligned}
				\norm{\avDkn - \bfDkn}\fs
				\leq {} & \lambda^2 \dkh{1 + \gamma} 
				\dkh{1 + 3 \eta^2 (L_g + \beta L_b)^2 / \gamma} \norm{\avDk - \bfDk}\fs \\
				& + C_2 \norm{\avXk - \bfXk}\fs + \dfrac{\eta}{8 \rho} \norm{H (\barXk)}\fs.
			\end{aligned}
		\end{equation*}
		Moreover, it follows from the condition $\eta \leq \bar{\eta}_9$ that the inequality
		\begin{equation*}
			\lambda^2 \dkh{1 + \gamma} 
			\dkh{1 + 3 \eta^2 (L_g + \beta L_b)^2 / \gamma}
			\leq \dfrac{5 + 3 \lambda^2}{8}
		\end{equation*}
		holds. The proof is completed. 
	\end{proof}

	Next, we evaluate the descent property of the sequence $\{ h(\barXk)\}$.
	
	\begin{lemma}\label{le:des-h}
		Suppose the conditions in Assumptions \ref{asp:smooth}, \ref{asp:network}, 
		and \ref{asp:parameters} hold.
		Let $\{\dkh{\bfXk, \bfDk}\}$ be the iterate sequence 
		generated by Algorithm \ref{alg:DESTINY} with $X_{\mathrm{initial}} \in \cR$.
		Then for any $k \in \bN$, it holds that
		\begin{equation*}
			h (\barXkn) \leq h (\barXk) 
			+ \dfrac{1 - \lambda^2}{8} \norm{\avXk - \bfXk}\fs
			+ \dfrac{21}{2} \eta L_f^2 \norm{\barXk\zz \barXk - I_p}\fs
			- \dfrac{5}{8} \eta \norm{H (\barXk)}\fs.
		\end{equation*}
	\end{lemma}
	
	\begin{proof}
		According to Proposition \ref{prop:bound}, 
		we know that the inequality $\norm{\bfXk}\ff \leq \sqrt{7dp / 6} + \sqrt{d}$ 
		holds for any $k \in \bN$.
		Then it follows from the local Lipschitz continuity of $h$ that
		\begin{equation*}
			\begin{aligned}
				h (\barXkn) 
				& = h (\barXk - \eta \barDk)
				\leq h (\barXk) - \eta \jkh{\nabla h (\barXk), \barDk} 
				+ \dfrac{1}{2} \eta^2 \dkh{L_g + \beta L_b} \norm{\barDk}\fs \\
				& = h (\barXk) - \eta \jkh{\nabla h (\barXk) - H (\barXk), \barDk} 
				- \eta \jkh{H (\barXk), \barDk} 
				+ \dfrac{1}{2} \eta^2 \dkh{L_g + \beta L_b} \norm{\barDk}\fs \\
				& \leq h (\barXk) - \eta \jkh{\nabla h (\barXk) - H (\barXk), \barDk} 
				- \eta \jkh{H (\barXk), \barDk} 
				+  \dfrac{1}{8} \eta \norm{H (\barXk)}\fs \\
				& \quad + \dfrac{1 - \lambda^2}{24} \norm{\avXk - \bfXk}\fs, \\
			\end{aligned}
		\end{equation*}
		where the last inequality follows from \eqref{eq:dk} and the condition 
		$\eta \leq \min \{\bar{\eta}_7, \bar{\eta}_{10}\}$.
		Moreover, we have
		\begin{equation*}
			\begin{aligned}
				& \abs{ \jkh{\nabla h (\barXk) - H (\barXk), \barDk} }
				\leq 4 \norm{\nabla h (\barXk) - H (\barXk)}\fs 
				+ \dfrac{1}{16} \norm{\barDk}\fs \\
				& = 9 \norm{\nabla f (\barXk) - \nabla f (\barXk \barXk\zz \barXk)}\fs 
				+ \dfrac{1}{16} \norm{\barDk}\fs \\
				& \leq 9 L_f^2 \norm{\barXk \dkh{\barXk\zz \barXk - I_p}}\fs 
				+ \dfrac{1}{16} \norm{\barDk}\fs
				\leq \dfrac{21}{2} L_f^2 \norm{\barXk\zz \barXk - I_p}\fs 
				+ \dfrac{1}{16} \norm{\barDk}\fs \\
				& \leq \dfrac{21}{2} L_f^2 \norm{\barXk\zz \barXk - I_p}\fs 
				+ \dfrac{1 - \lambda^2}{24 \eta} \norm{\avXk - \bfXk}\fs
				+ \dfrac{1}{8} \norm{H (\barXk)}\fs,
			\end{aligned}
		\end{equation*}
		where the last inequality results from \eqref{eq:dk} and the condition
		$\eta \leq \bar{\eta}_{11}$. 
		And straightforward manipulations give us
		\begin{equation*}
				\jkh{H (\barXk), \barDk}
				= \jkh{H (\barXk), \barDk - H (\barXk) + H (\barXk)}
				= \jkh{H (\barXk), \barDk - H (\barXk)} + \norm{H (\barXk)}\fs,
		\end{equation*}
		and 
		\begin{equation*}
			\begin{aligned}
				\abs{ \jkh{H (\barXk), \barDk - H (\barXk)} }
				& \leq \dfrac{1}{8} \norm{H (\barXk)}\fs + 2\norm{\barDk - H (\barXk)}\fs \\
				& \leq \dfrac{1}{8} \norm{H (\barXk)}\fs 
				+ \dfrac{1 - \lambda^2}{24 \eta} \norm{\avXk - \bfXk}\fs,
			\end{aligned}
		\end{equation*}
		where the last inequality is implied by \eqref{eq:diff-hd} and the condition 
		$\eta \leq \bar{\eta}_{11}$. 
		Combing the above relationships, 
		we finally arrive at the assertion of this lemma. 
	\end{proof}

	Now we are in the position to introduce the merit function
	\begin{equation*}
		\tilde{h}(\bfX, \bfD) := h ((\bfone_d \otimes I_n)\zz \bfX) + \norm{\bfJ \bfX - \bfX}\fs
		+ \rho \norm{\bfJ \bfD - \bfD}\fs,
	\end{equation*}
	where $\rho > 0$ is a constant defined in \eqref{eq:constants-2}.
	
	\begin{proposition}\label{prop:des-L}
		Suppose the conditions in Assumptions \ref{asp:smooth}, \ref{asp:network}, 
		and \ref{asp:parameters} hold.
		Let $\{\dkh{\bfXk, \bfDk}\}$ be the iterate sequence 
		generated by Algorithm \ref{alg:DESTINY} with $X_{\mathrm{initial}} \in \cR$.
		Then for any $k \in \bN$, the following sufficient descent condition holds,
		which implies that the sequence $\{ \tilde{h} (\bfXk, \bfDk)\}$ 
		is monotonically non-increasing.
		\begin{equation} \label{eq:des-merit}
			\begin{aligned}
				\tilde{h}(\bfXkn, \bfDkn) 
				\leq {} & \tilde{h}(\bfXk, \bfDk) 
				- \dfrac{3}{16} \dkh{1 - \lambda^2} \norm{\avXk - \bfXk}\fs
				- \dfrac{3 \rho}{16} \dkh{1 - \lambda^2} \norm{\avDk - \bfDk}\fs \\
				& - \dfrac{\eta}{2} \norm{G (\barXk)}\fs
				- \dfrac{\eta}{2} L_f^2 \norm{\barXk\zz \barXk - I_p}\fs.
			\end{aligned}
		\end{equation}
	\end{proposition}

	\begin{proof}
		Combing Lemmas \ref{le:consensus-error}, \ref{le:tracking-error} and \ref{le:des-h},
		we have
		\begin{equation*}
			\begin{aligned}
				\tilde{h}(\bfXkn, \bfDkn) 
				\leq {} & \tilde{h}(\bfXk, \bfDk) 
				- \dfrac{3}{16} \dkh{1 - \lambda^2} \norm{\avXk - \bfXk}\fs
				 - \dfrac{\eta}{2} \norm{H (\barXk)}\fs \\
				& - \rho \dkh{ \dfrac{3}{8} \dkh{1 - \lambda^2} - \dfrac{\eta^2 C_1}{\rho} } \norm{\avDk - \bfDk}\fs
				+ \dfrac{21}{2} \eta L_f^2 \norm{\barXk\zz \barXk - I_p}\fs \\
				\leq {} & \tilde{h}(\bfXk, \bfDk) 
				- \dfrac{3}{16} \dkh{1 - \lambda^2} \norm{\avXk - \bfXk}\fs
				- \dfrac{3 \rho}{16} \dkh{1 - \lambda^2} \norm{\avDk - \bfDk}\fs \\
				& - \dfrac{\eta}{2} \norm{H (\barXk)}\fs
				+ \dfrac{21}{2} \eta L_f^2 \norm{\barXk\zz \barXk - I_p}\fs,
			\end{aligned}
		\end{equation*}
		where the last inequality follows from the condition $\eta \leq \bar{\eta}_{12}$.
		Together with Lemma \ref{le:expen} and $\beta \geq 22 L_f^2$,
		we can obtain the sufficient descent condition \eqref{eq:des-merit},
		which infers that
		\begin{equation*}
			\tilde{h}(\bfXkn, \bfDkn) \leq \tilde{h}(\bfXk, \bfDk).
		\end{equation*}
		The proof is completed.
		\end{proof}
	
\subsection{Global convergence}

	Finally, we can establish the global convergence of Algorithm \ref{alg:DESTINY}
	together with the worst case complexity.
	
	\begin{theorem}
		Suppose the conditions in Assumptions \ref{asp:smooth}, \ref{asp:network}, 
		and \ref{asp:parameters} hold.
		Let $\{\dkh{\bfXk, \bfDk}\}$ be the iterate sequence 
		generated by Algorithm \ref{alg:DESTINY} with $X_{\mathrm{initial}} \in \cR$.
		Then $\{\dkh{\bfXk, \bfDk}\}$ has at least one accumulation point.
		Moreover, for any accumulation point $\dkh{\bfX^{\ast}, \bfD^{\ast}}$,
		we have $\bfX^{\ast} = (\bfone_d \otimes I_n) \barX^{\ast}$,
		where $\barX^{\ast} \in \Snp$ is a first-order stationary point 
		of the problem \eqref{eq:opt-stiefel}.
		Finally, the following relationships hold,
		which results in the global sublinear convergence rate of Algorithm \ref{alg:DESTINY}.
		\begin{equation} \label{eq:sublinear-consensus}
			\min_{k = 0, 1, \dotsc, K - 1} \norm{\avXk - \bfXk}\fs
			\leq \dfrac{16 (\tilde{h}(\bfX_0, \bfD_0) - \underline{h})}
			{3 \dkh{1 - \lambda^2} K},
		\end{equation}
		\begin{equation} \label{eq:sublinear-tracking}
			\min_{k = 0, 1, \dotsc, K - 1} \norm{G (\barXk)}\fs
			\leq \dfrac{2 (\tilde{h}(\bfX_0, \bfD_0) - \underline{h})}{\eta K},
		\end{equation}
		and
		\begin{equation} \label{eq:sublinear-gradient}
			\min_{k = 0, 1, \dotsc, K - 1} \norm{\barXk\zz \barXk - I_p}\fs
			\leq \dfrac{2 (\tilde{h}(\bfX_0, \bfD_0) - \underline{h})}{\eta L_f^2 K},
		\end{equation}
		where $\underline{h}$ is a constant.
	\end{theorem}
	
	\begin{proof}
		According to Lemma \ref{prop:bound},
		we know that the sequence $\{\dkh{\bfXk, \bfDk}\}$ is bounded.
		Hence, the lower boundedness of $\{ \tilde{h} (\bfXk, \bfDk)\}$
		is owing to the continuity of $\tilde{h}$.
		Namely, there exists a constant $\underline{h}$ such that
		$\tilde{h} (\bfXk, \bfDk) \geq \underline{h}$ for any $k \in \bN$.
		Then, it follows from Proposition \ref{prop:des-L} that
		the sequence $\{ \tilde{h} (\bfXk, \bfDk)\}$ is convergent and 
		the following relationships hold.		
		\begin{equation}\label{eq:limit}
			\lim\limits_{k \to \infty} \norm{\avXk - \bfXk}\fs = 0, \quad
			\lim\limits_{k \to \infty} \norm{G (\barXk)}\fs = 0, \quad
			\lim\limits_{k \to \infty} \norm{\barXk\zz \barXk - I_p}\fs = 0.
		\end{equation}
		According to the Bolzano-Weierstrass theorem, it follows that 
		$\{\dkh{\bfXk, \bfDk}\}$  exists an accumulation point, 
		say $\dkh{\bfX^{\ast}, \bfD^{\ast}}$.
		The relationships in \eqref{eq:limit} imply that
		$\bfX^{\ast} = (\bfone_d \otimes I_n) \barX^{\ast}$
		for some $\barX^{\ast} \in \Snp$.
		Moreover, we have
		\begin{equation*}	
			\grad f (\barX^{\ast}) = G (\bar{X}^{\ast})
			= \lim\limits_{k \to \infty} G (\barXk) = 0,
		\end{equation*}
		which implies that $\barX^{\ast}$  is a first-order stationary point of \eqref{eq:opt-stiefel}.
		Finally, we prove that the relationships 
		\eqref{eq:sublinear-consensus}-\eqref{eq:sublinear-gradient} hold.
		Indeed, it follows from Proposition \ref{prop:des-L} that
		\begin{equation*}
			\begin{aligned}
				\sum_{k = 0}^{K - 1} \norm{\avXk - \bfXk}\fs 
				& \leq \dfrac{16}{3 \dkh{1 - \lambda^2}} 
				\sum_{k = 0}^{K - 1} \dkh{\tilde{h}(\bfXk, \bfDk) - \tilde{h}(\bfXkn, \bfDkn)} \\
				& \leq \dfrac{16 (\tilde{h}(\bfX_0, \bfD_0) - \tilde{h}(\bfX_K, \bfD_K))}
				{3 \dkh{1 - \lambda^2}}
				\leq \dfrac{16 (\tilde{h}(\bfX_0, \bfD_0) - \underline{h})}
				{3 \dkh{1 - \lambda^2}},
			\end{aligned}
		\end{equation*}
		which implies the relationship \eqref{eq:sublinear-consensus}.
		The other two relationships can be proved similarly.
		Therefore, we complete the proof.
	\end{proof}

\section{Numerical Simulations}\label{sec:numerical}

	In this section, the numerical performance of DESTINY is evaluated 
	through comprehensive numerical experiments in comparison to state-of-the-art algorithms.
	The corresponding experiments are performed on a workstation
	with two Intel Xeon Gold 6242R CPU processors (at $3.10$GHz$\times 20 \times 2$) 
	and 510GB of RAM under Ubuntu 20.04.
	All the tested algorithms are implemented in the \texttt{Python} language,
	and the communication is realized via the package \texttt{mpi4py}.
	
\subsection{Test problems}\label{subsec:test-problems}

	In the numerical experiments, we test three types of problems, 
	including PCA, OLSR and SDL, 
	which are introduced in Subsection \ref{subsec:broad-applications}.
	
	In the PCA problem \eqref{eq:opt-pca},
	we construct the global data matrix $A\in\Rnm$ 
	(assuming $n\leq m$ without loss of generality)
	by its (economy-form) singular value decomposition as follows.
	\begin{equation}\label{eq:gen-A}
		A = U \Sigma V\zz,
	\end{equation}
	where both $U \in \Rnn$ and $V \in \Rmn$ are orthogonal matrices 
	orthonormalized from randomly generated matrices,
	and $\Sigma \in \Rnn$ is a diagonal matrix with diagonal entries 
	\begin{equation}\label{eq:Sigma_ii}
		\Sigma_{ii}=\xi^{i / 2}, \quad \iin,
	\end{equation}  
	for a parameter $\xi \in (0, 1)$ that determines the decay rate of the singular values of $A$.
	In general, smaller decay rates (with $\xi$ closer to 1) correspond to more difficult cases. 
	Finally, the global data matrix $A$ is uniformly distributed into $d$ agents.
	
	As for the OLSR problem \eqref{eq:opt-olsr} and SDL problem \eqref{eq:opt-sdl},
	we use real-world text and image datasets to construct the global data matrix, respectively.
	
	Unless otherwise specified, the network is randomly generated 
	based on the Erdos-Renyi graph with a fixed probability $\mathtt{prob} \in (0, 1)$
	in the following experiments.
	And we select the mixing matrix $W$ 
	to be the Metroplis constant edge weight matrix \cite{Shi2015}.
	
	In our experiments, we collect and compare two performance measurements:
	(relative) substationarity violation defined by
	\begin{equation*}
		\dfrac{1}{\norm{\grad f (\barX_0)}\ff} \norm{\grad f (\barXk)}\ff,
	\end{equation*}
	and consensus error defined by
	\begin{equation*}
		\sqrt{\dfrac{1}{d} \sumiid \norm{\Xik - \barXk}\fs}.
	\end{equation*}
	When presenting numerical results, we plot the decay of the above two measurements
	versus the round of communications.

\subsection{Default settings}\label{subsec:default-settings}

	In this subsection, we determine the default settings for our algorithm.
	The corresponding experiments are performed 
	on an Erdos-Renyi network with $\mathtt{prob} = 0.5$.
	
	In the first experiment, we test the performances of DESTINY
	with different choices of stepsizes.
	Classical strategies for choosing stepsizes, such as line search,
	are not applicable in the decentralized setting.
	Recently, BB stepsizes have been introduced to the decentralized optimization 
	for strongly convex problems \cite{Gao2022achieving}.
	In order to improve the performance of DESTINY, 
	we intend to apply the following BB stepsizes.
	\begin{equation*}
		\eta^\mathrm{BB}_{i, k} = \abs{\frac{\jkh{S_{i,k}, J_{i,k}}}{\jkh{J_{i,k}, J_{i,k}}}},
	\end{equation*}
	where $S_{i,k} = \Xik - \Xikp$, and $J_{i,k} = \Dik - \Dikp$.
	As illustrated in Figure \ref{fig:PCA_stepsize},
	the above BB stepsizes can significantly enhance the efficiency of DESTINY.
	Hence, in subsequent experiments,
	we equip DESTINY with BB stepsizes by default.
	To the best of our knowledge, this is the first attempt, for nonconvex optimization problems,
	to accelerate the convergence of decentralized algorithms by BB stepsizes.

	\begin{figure}[ht!]
		\centering
		
		\subfigure[Substationarity Violation]{
			\label{subfig:PCA_kkt_stepsize}
			\includegraphics[width=0.3\linewidth]{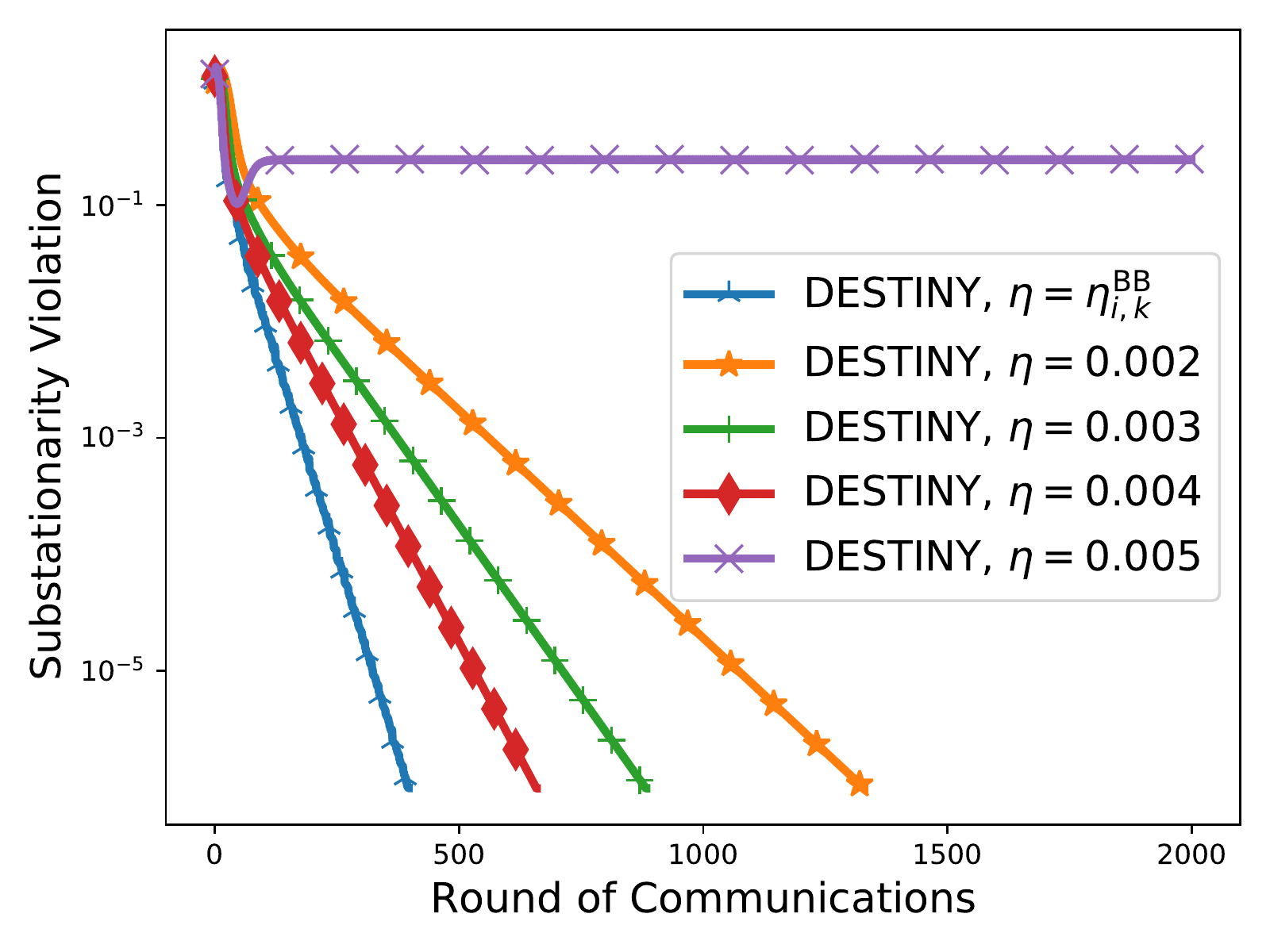}
		}
		\subfigure[Consensus Error]{
			\label{subfig:PCA_cons_stepsize}
			\includegraphics[width=0.3\linewidth]{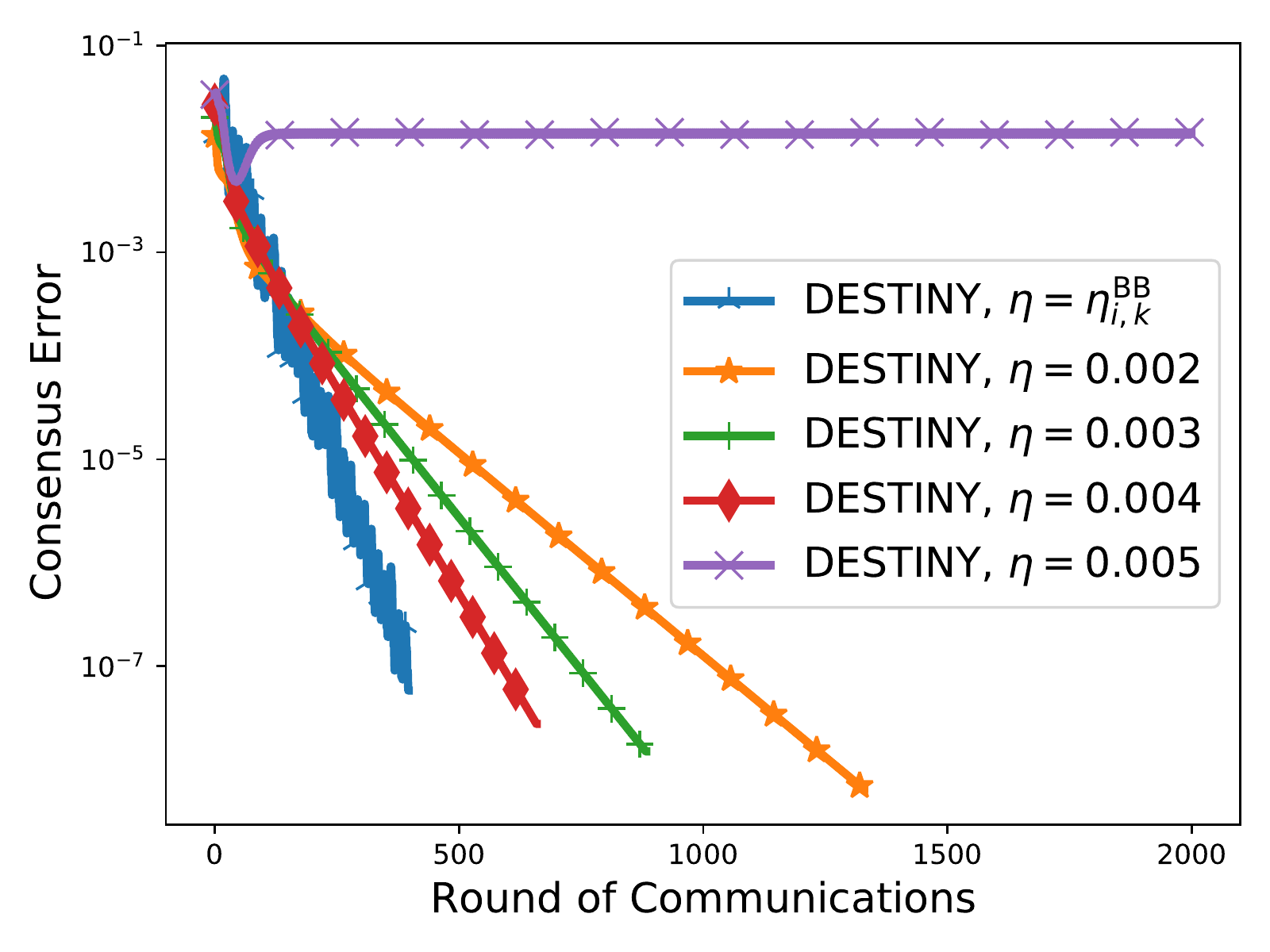}
		}
		
		\caption{Comparison of different stepsizes for DESTINY.
			The experimental settings are 
			$n = 100$, $m = 3200$, $p = 5$, $\xi = 0.9$, and $d = 32$.
			The penalty parameter of DESTINY is $\beta = 1$.}
		\label{fig:PCA_stepsize}
	\end{figure}

	Next, we present that our algorithm is not sensitive to the penalty parameter $\beta$.
	Figure \ref{fig:PCA_beta} depicts the performances of DESTINY 
	with different values of $\beta$ which are distinguished by colors.
	In Figure \ref{subfig:PCA_feas_beta}, the y-axis depicts the feasibility violation 
	in the logarithmic scale, which is defined by:
	\begin{equation*}
		\sqrt{\dfrac{1}{d} \sumiid \norm{\Xik\zz \Xik - I_p}\fs}.
	\end{equation*}
	We can observe that the curves of substationarity violations, 
	consensus errors, and feasibility violations almost coincide with each other,
	which indicates that DESTINY has almost the same performance
	in a wide range of penalty parameters.
	This observation validates the robustness of our algorithm to penalty parameters.
	Therefore, we fix $\beta = 1$ by default in the subsequent experiments.

	\begin{figure}[ht!]
		\centering
		
		\subfigure[Substationarity Violation]{
			\label{subfig:PCA_kkt_beta}
			\includegraphics[width=0.3\linewidth]{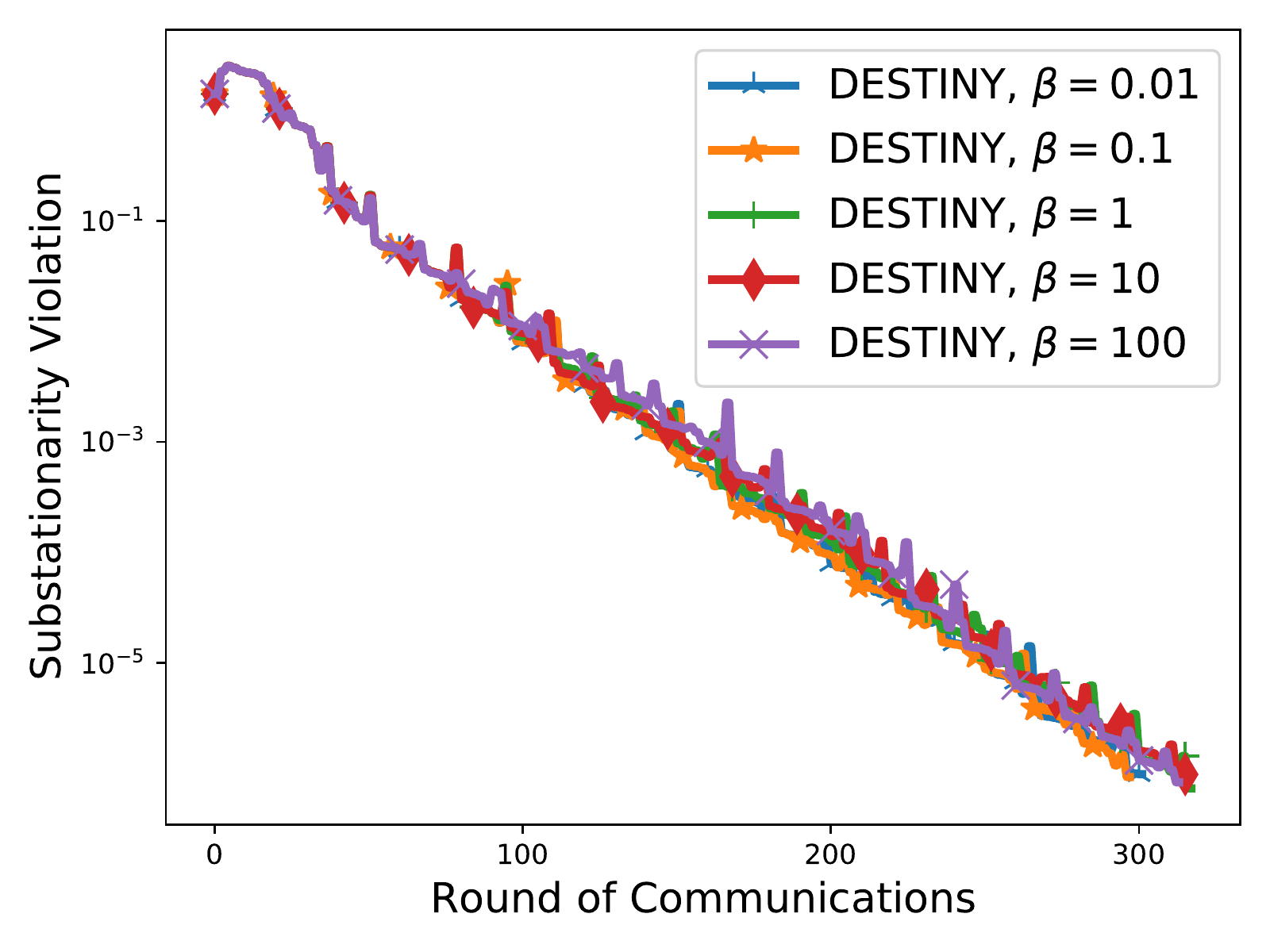}
		}
		\subfigure[Consensus Error]{
			\label{subfig:PCA_cons_beta}
			\includegraphics[width=0.3\linewidth]{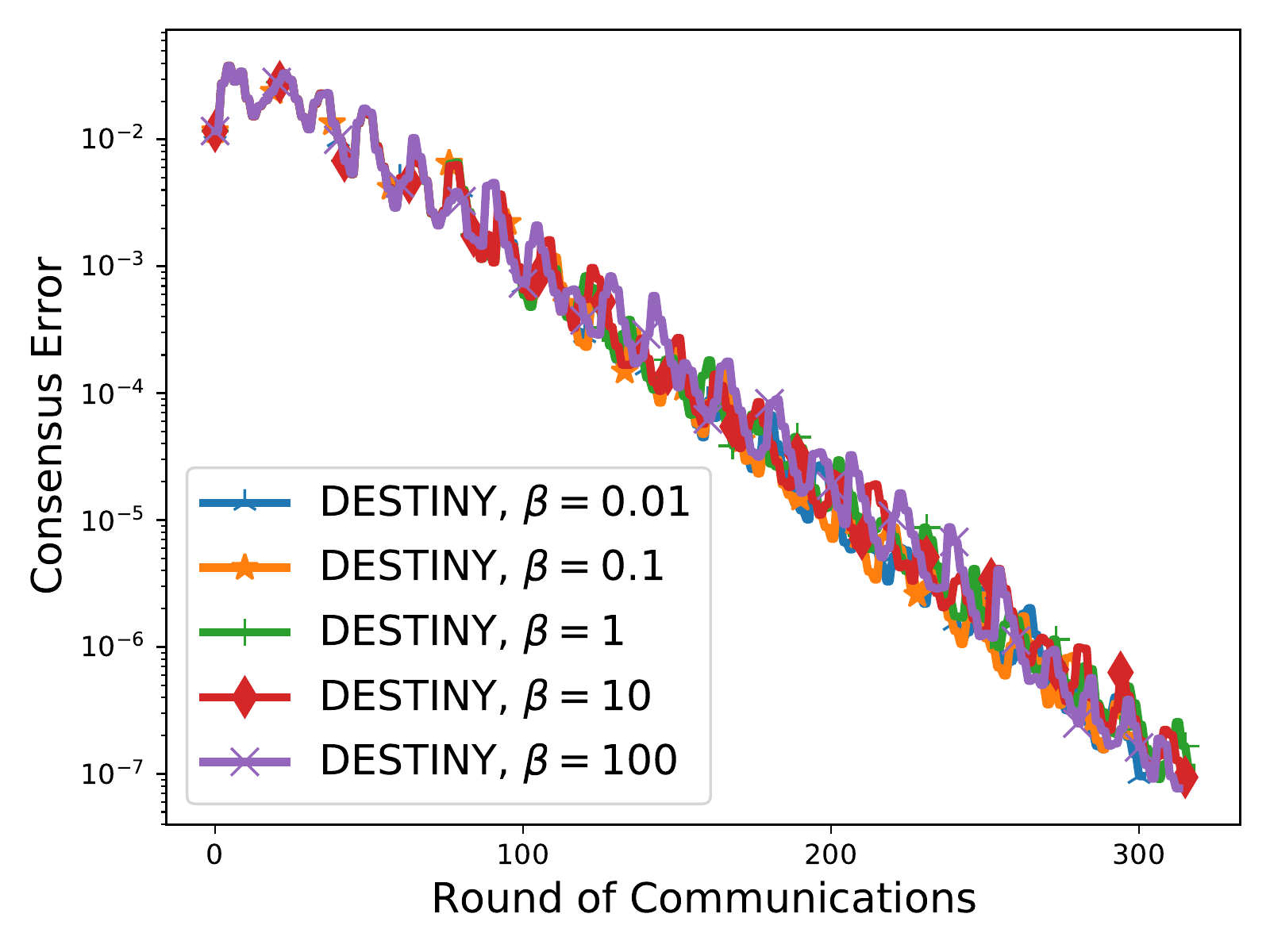}
		}
		\subfigure[Feasibility Violation]{
			\label{subfig:PCA_feas_beta}
			\includegraphics[width=0.3\linewidth]{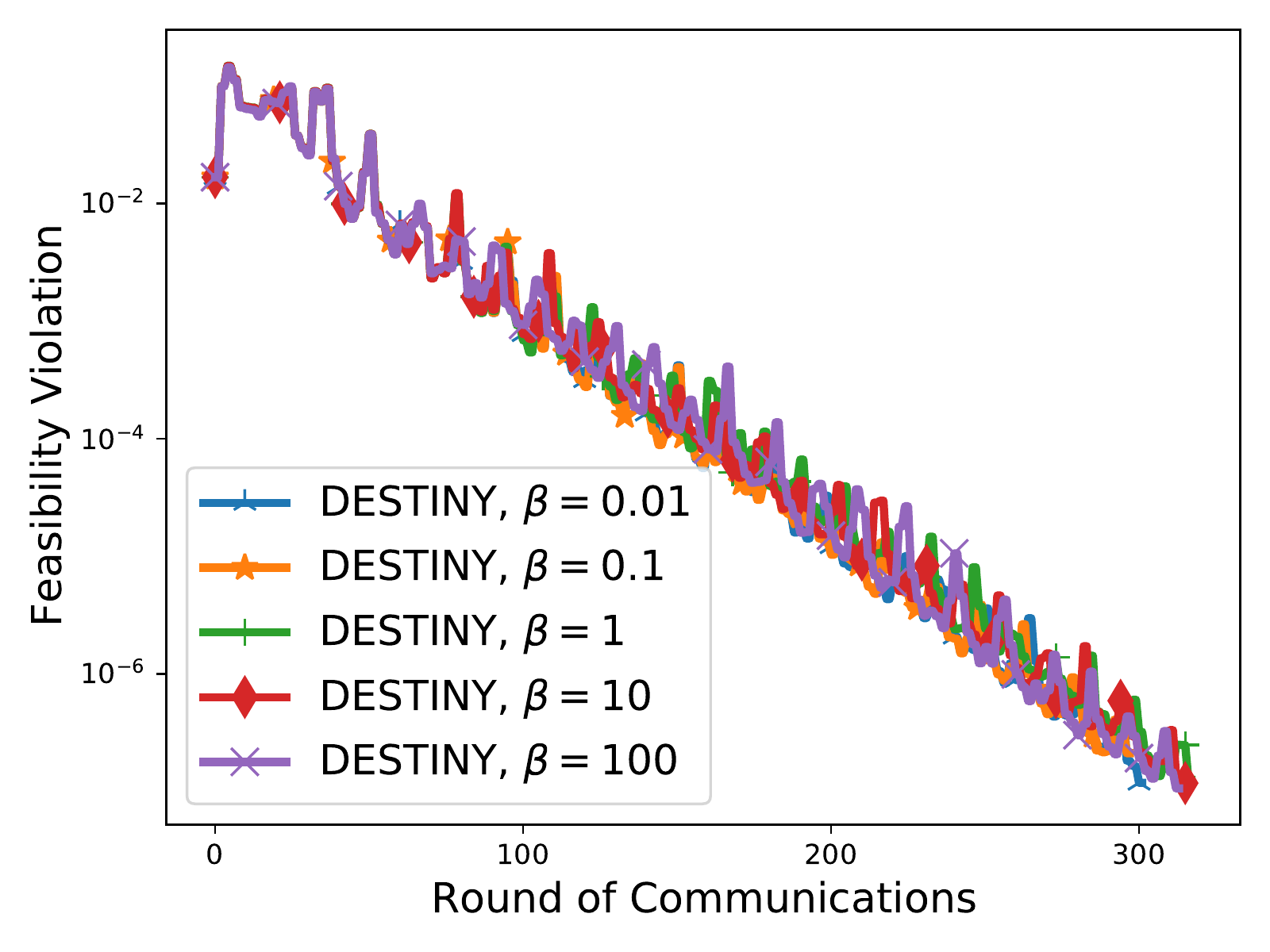}
		}
			
		\caption{Comparison of different penalty parameters for DESTINY.
			The experimental settings are
			$n = 200$, $m = 6400$, $p = 5$, $\xi = 0.8$, and $d = 16$.}
		\label{fig:PCA_beta}
	\end{figure}
	
\subsection{Comparison on decentralized PCA}

	In this subsection, we compare the performance of DESTINY 
	in solving decentralized PCA problems \eqref{eq:opt-pca} with some other state-of-the-art methods,
	including ADSA\cite{Gang2019}, DRGTA \cite{Chen2021decentralized}, and DeEPCA \cite{Ye2021}.
	Among these three algorithms, ADSA and DRGTA require to choose the stepsize.
	In the empirical experiments,
	we find that DRGTA also benefits from the BB stepsizes,
	but ADSA can not converge with it.
	Hence, DRGTA is equipped with the same BB stepsizes as DESTINY,
	while the stepsize of ADSA is fixed as a hand-optimized constant.
	In the following experiments, for fair comparisons,
	both DRGTA and DeEPCA perform only one round of communications per iteration.
		
	In order to assess the quality of the above four solvers,
	we perform a set of experiments on randomly generated matrices
	as described in Subsection \ref{subsec:test-problems}.
	The network is built based on an Erdos-Renyi graph
	with $\mathtt{prob}$ ranging from $0.2$ to $0.6$ with increment $0.2$.
	Other experimental settings are $n = 300$, $m = 3200$, $p = 10$, $\xi = 0.95$, and $d = 32$.
	The numerical results are illustrated in Figure \ref{fig:PCA}.
	It can be observed that DeEPCA and DRGTA can not converge when $\mathtt{prob} = 0.2$,
	which signifies that a single consensus step per iteration
	is not sufficient for these two algorithms to converge in sparse networks.
	Apart from this special case, DeEPCA has the best performance among all the algorithms.
	Moreover, DESTINY is competitive with ADSA and significantly outperforms DRGTA in all cases.
	It should be noted that DeEPCA and ADSA are tailored for decentralized PCA problems
	and can not be extended to generic cases.
	
	\begin{figure}[ht!]
		\centering
		
		\subfigure[$\mathtt{prob} = 0.2$]{
			\label{subfig:PCA_2_kkt}
			\includegraphics[width=0.3\linewidth]{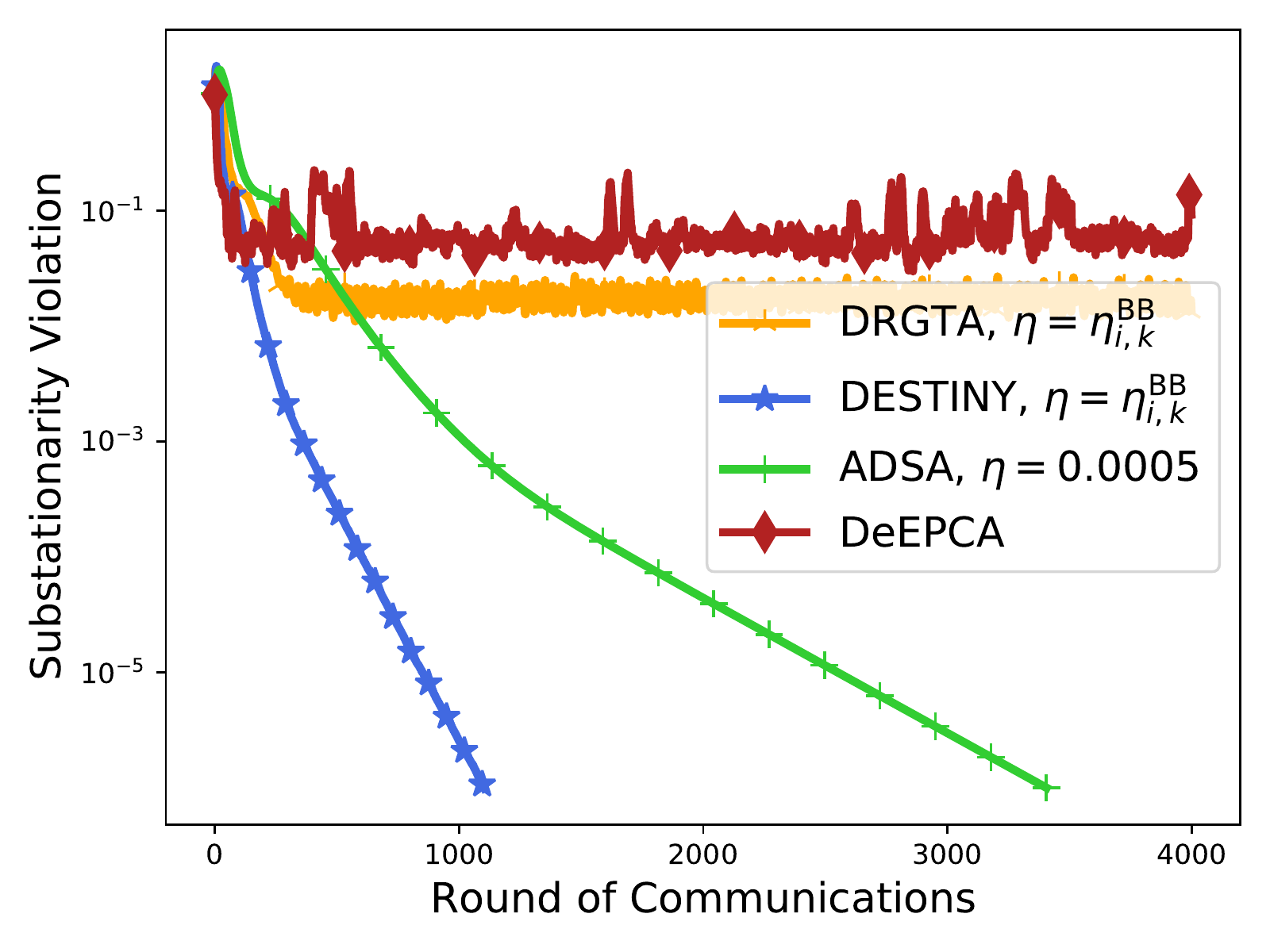}
		}
		\subfigure[$\mathtt{prob} = 0.4$]{
			\label{subfig:PCA_4_kkt}
			\includegraphics[width=0.3\linewidth]{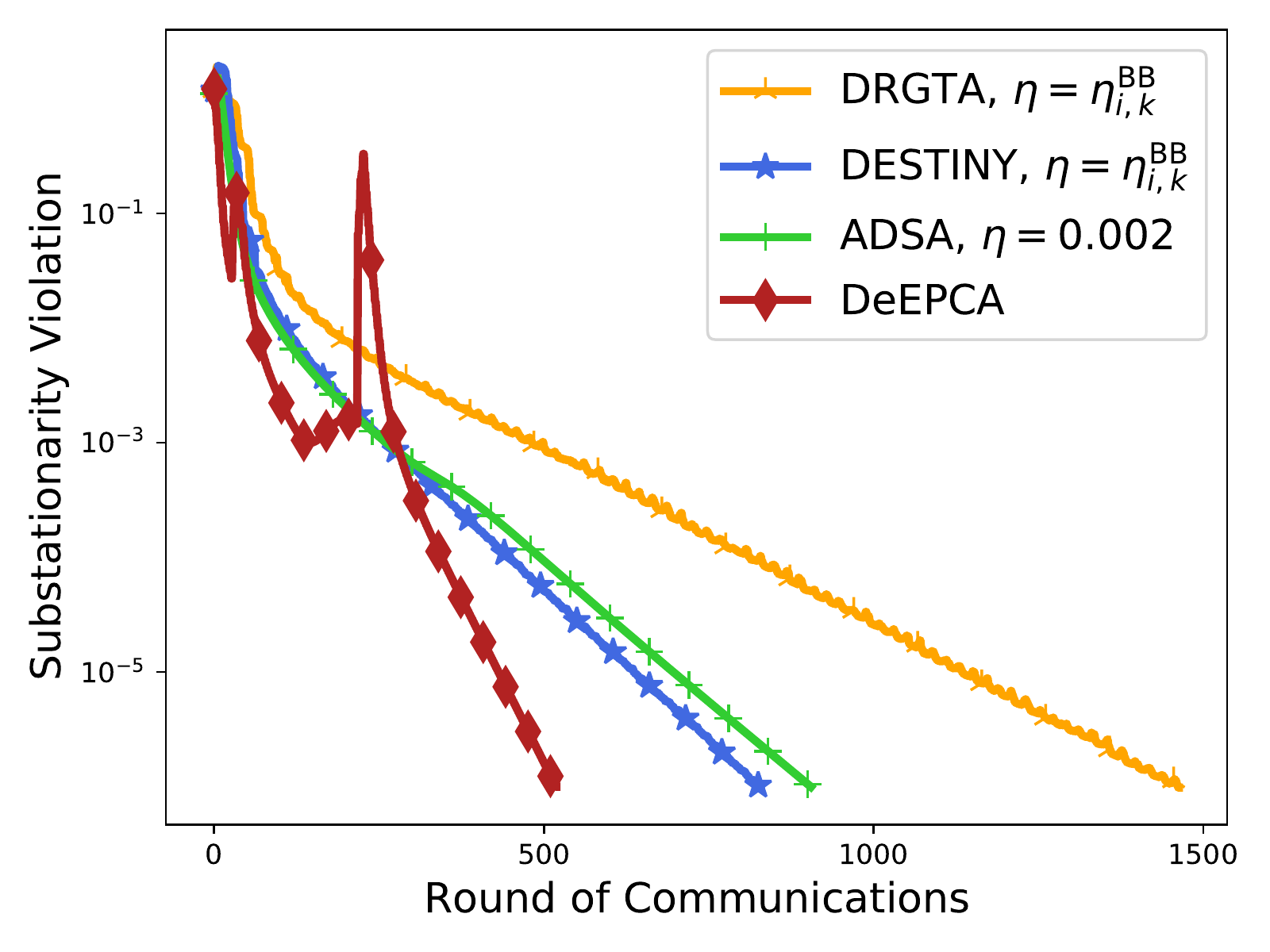}
		}
		\subfigure[$\mathtt{prob} = 0.6$]{
			\label{subfig:PCA_6_kkt}
			\includegraphics[width=0.3\linewidth]{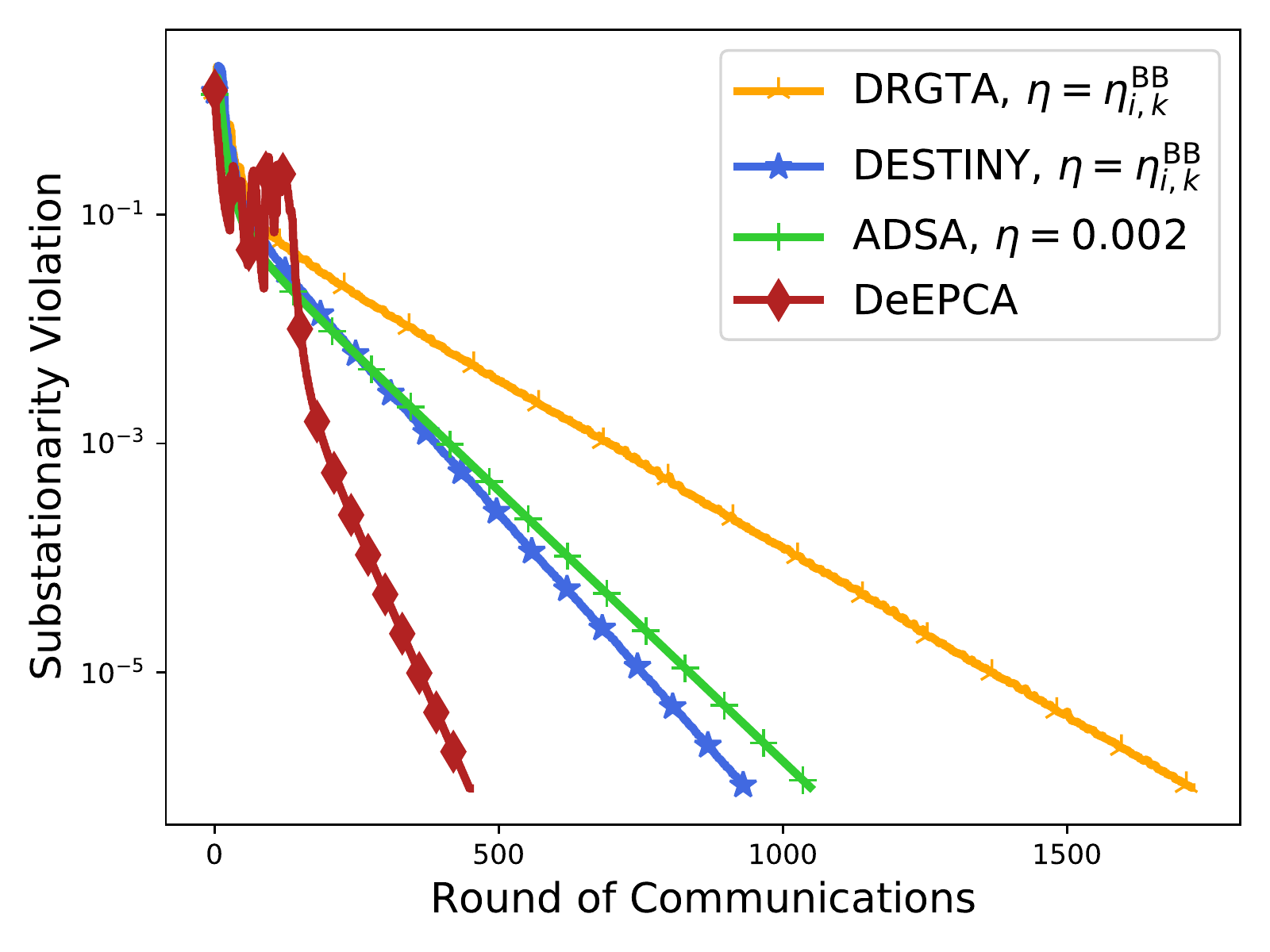}
		}

		\subfigure[$\mathtt{prob} = 0.2$]{
			\label{subfig:PCA_2_cons}
			\includegraphics[width=0.3\linewidth]{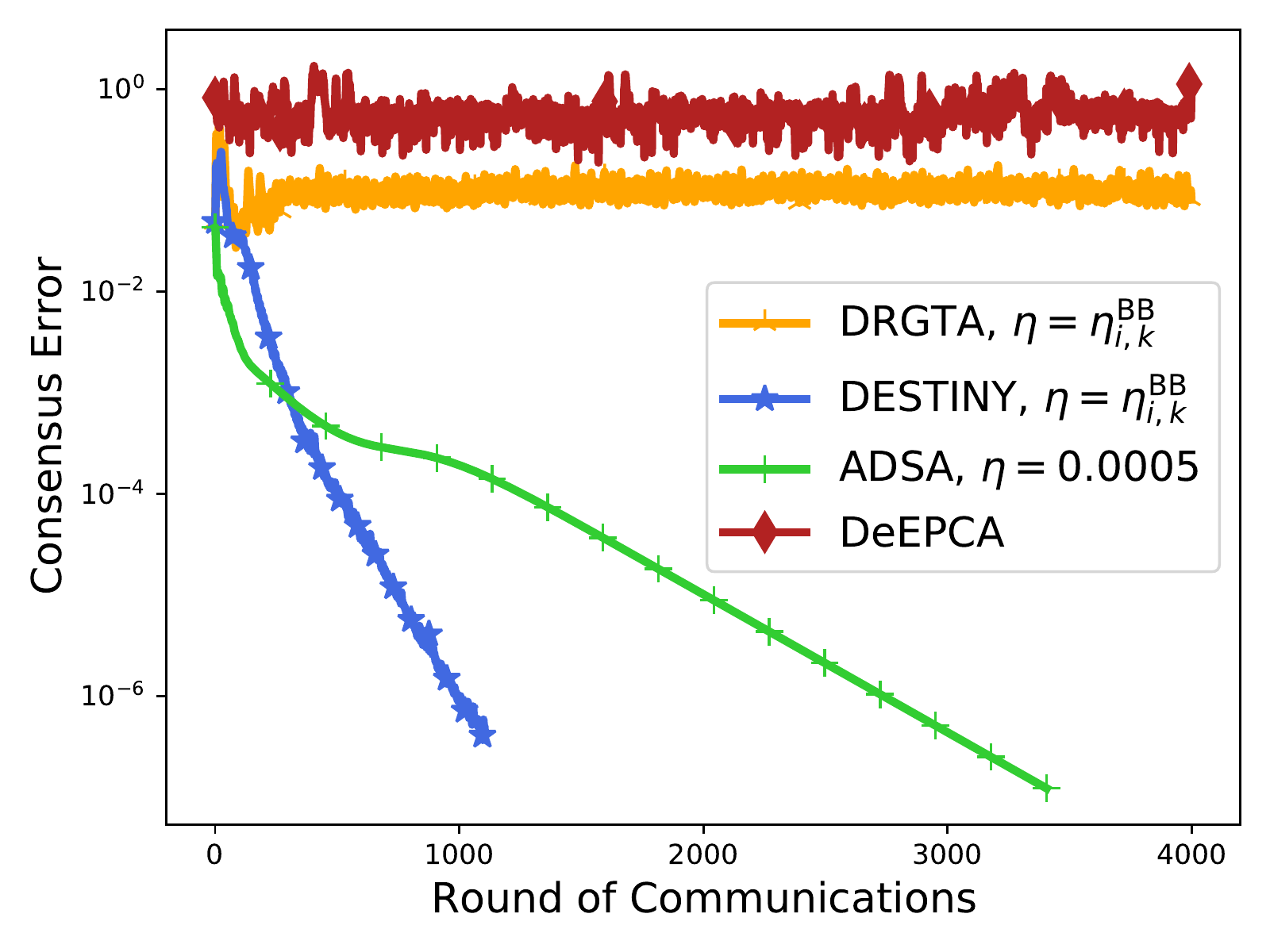}
		}
		\subfigure[$\mathtt{prob} = 0.4$]{
			\label{subfig:PCA_4_cons}
			\includegraphics[width=0.3\linewidth]{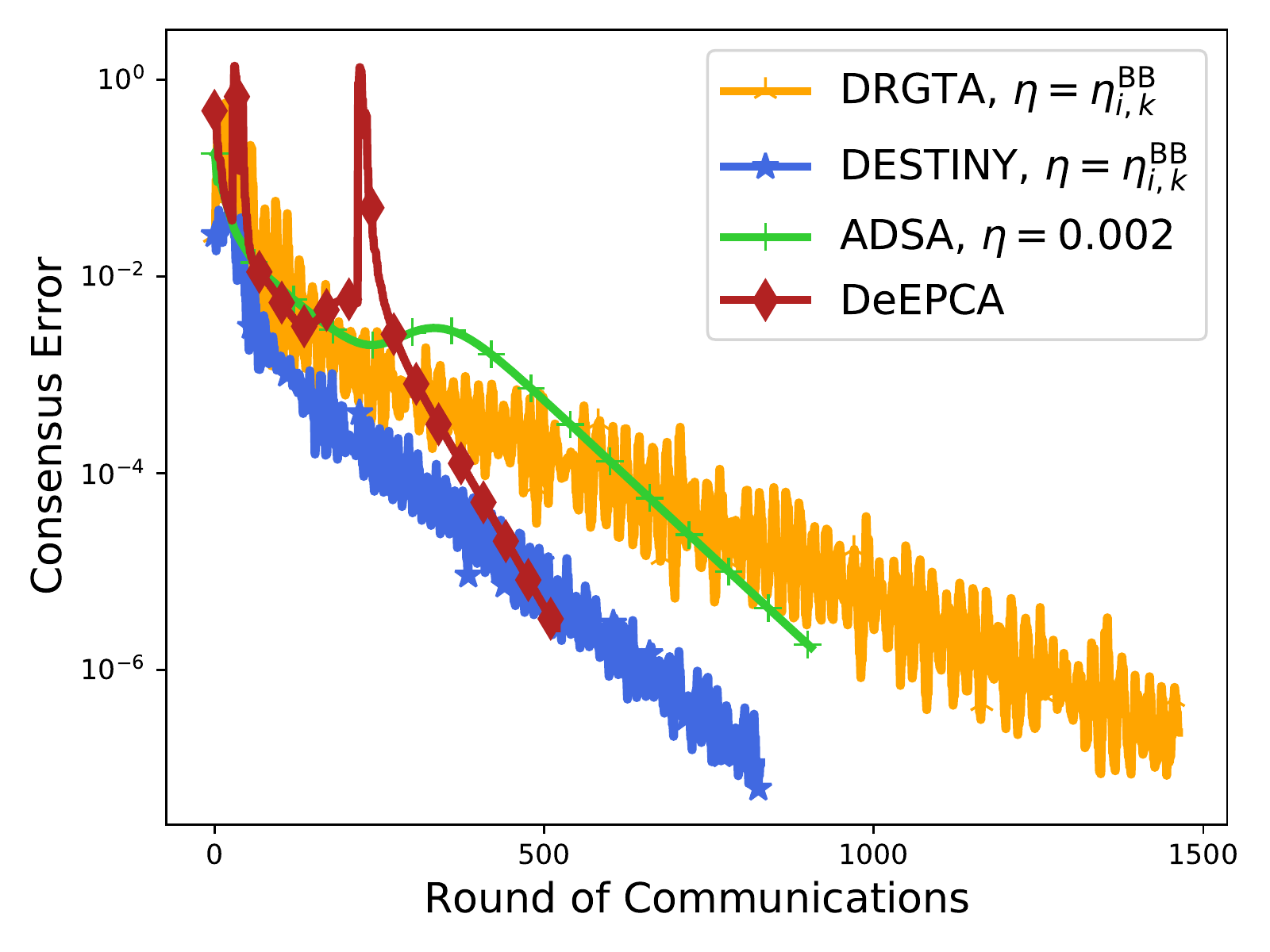}
		}
		\subfigure[$\mathtt{prob} = 0.6$]{
			\label{subfig:PCA_6_cons}
			\includegraphics[width=0.3\linewidth]{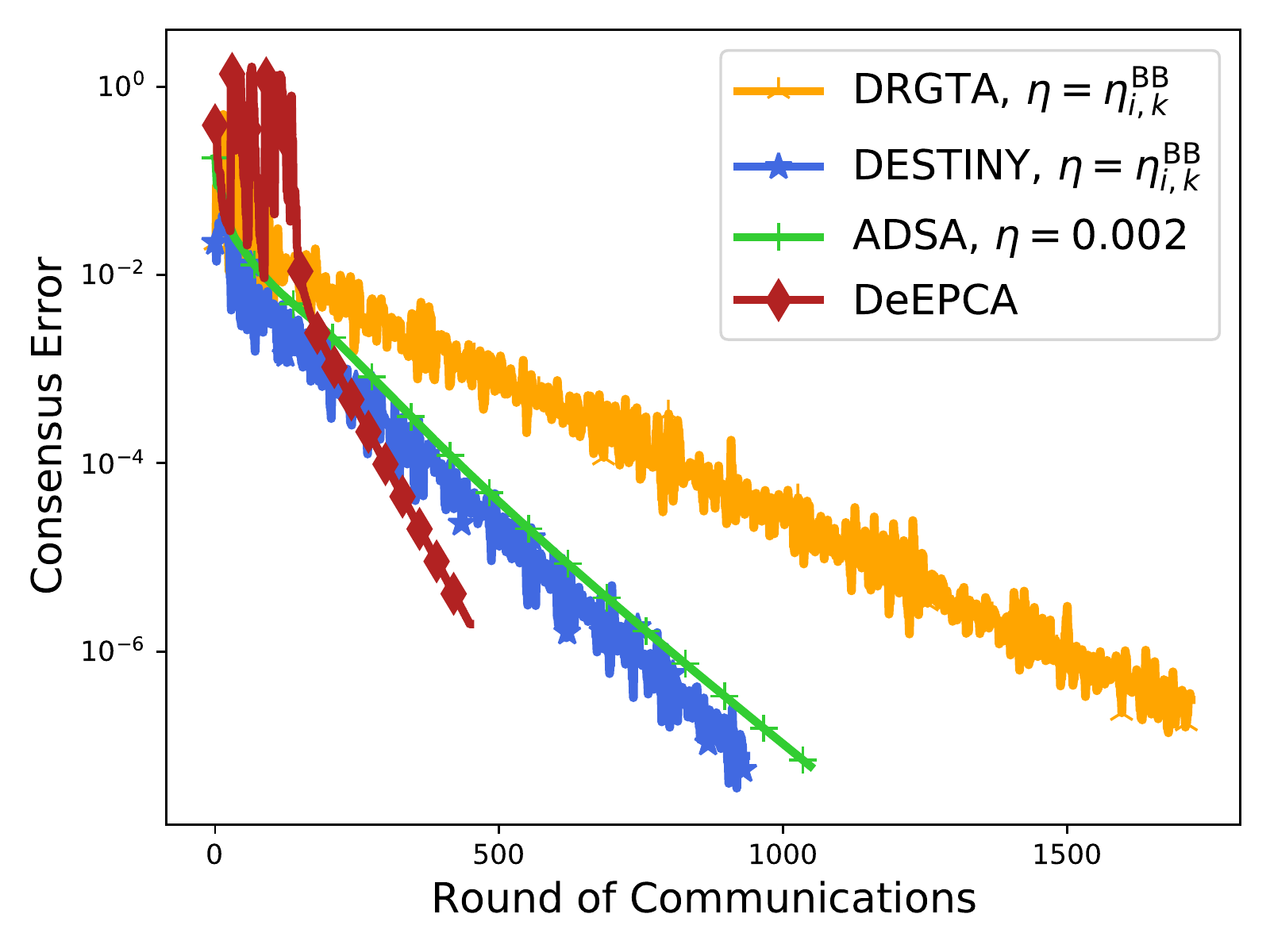}
		}
		
		\caption{Comparison of DRGTA, DESTINY, ADSA, and DeEPCA
			for different values of $\mathtt{prob}$ in solving decentralized PCA problems. 
			The figures in the first and second rows 
			depict substationarity violations and consensus errors, respectively.}
		\label{fig:PCA}
	\end{figure}

\subsection{Comparison on decentralized OLSR}

	In this subsection, the abilities of DESTINY and DRGTA \cite{Chen2021decentralized}
	are examined in solving decentralized OLSR problems \eqref{eq:opt-olsr}
	on the real-world text dataset Cora \cite{Mccallum2000automating}.
	For our testing in this case, we select a subset of Cora
	and extract the first $n = 100$ features to use,
	which contains research papers from three subfields, 
	including data structure (DS), hardware and architecture (HA), 
	and programming language (PL).
	The number of samples $(m)$ and classes $(p)$ are summarized as follows.
		
	\begin{itemize}
		
		\item Cora-DS: $m = 751$, $p = 9$.
		
		\item Cora-HA: $m = 400$, $p = 7$.
		
		\item Cora-PL: $m = 1575$, $p = 9$.
		
	\end{itemize}

	The detailed descriptions and download links of this dataset
	can be found in \cite{Zhang2020eigenvalue}.
	The number of agents is set to $d = 32$ in this experiment.
	We provide the numerical results in Figure \ref{fig:OLSR},
	which demonstrates that DESTINY attains better performances than DRGTA 
	in terms of both substationarity violations and consensus errors.
	Therefore, our algorithm manifests its efficiency in dealing with real-world datasets.
	
	\begin{figure}[ht!]
		\centering
		
		\subfigure[Cora-DS]{
			\label{subfig:OLSR_Cora_DS_kkt} 
			\includegraphics[width=0.3\linewidth]{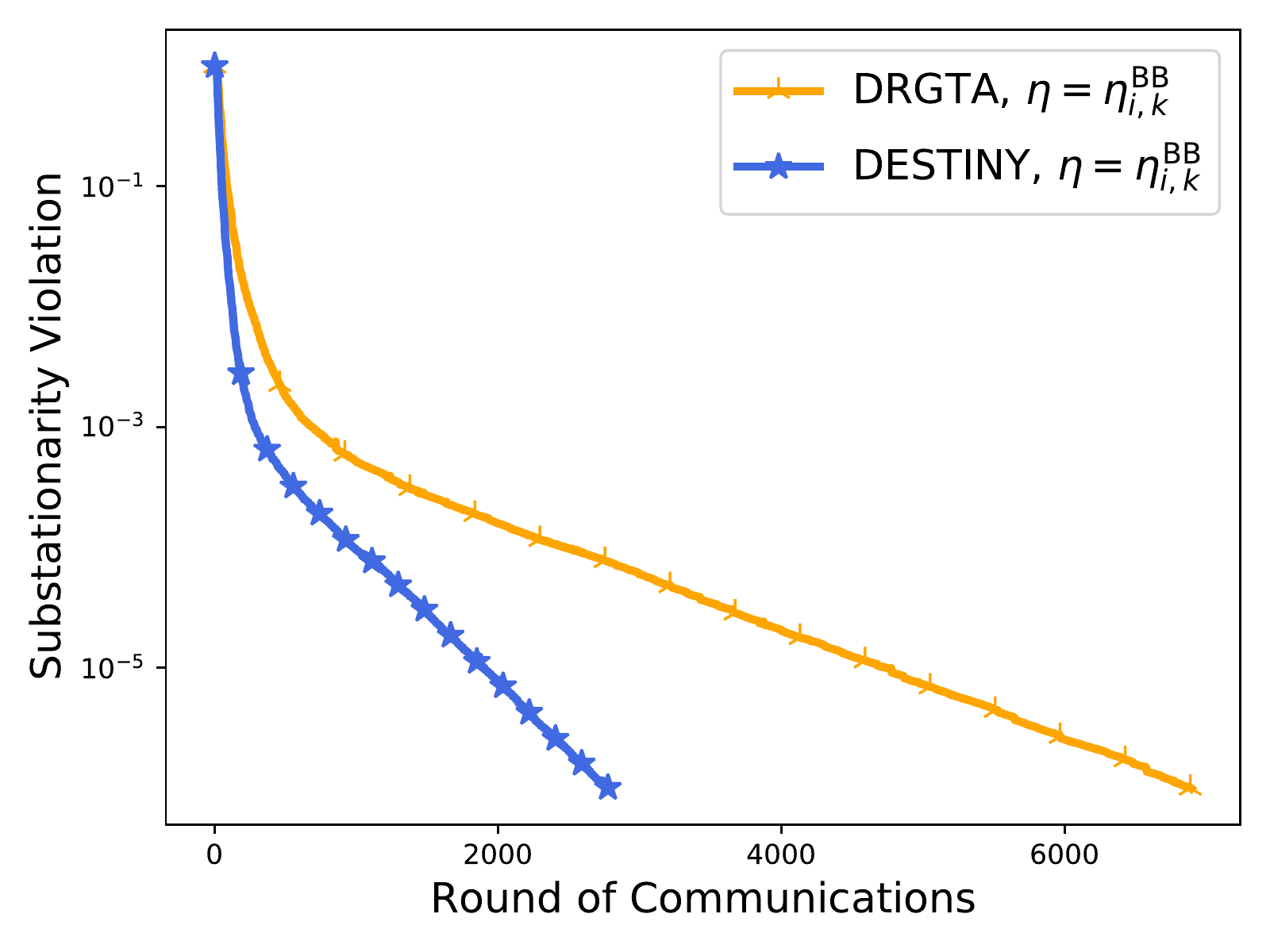}
		}
		\subfigure[Cora-HA]{
			\label{subfig:OLSR_Cora_HA_kkt}
			\includegraphics[width=0.3\linewidth]{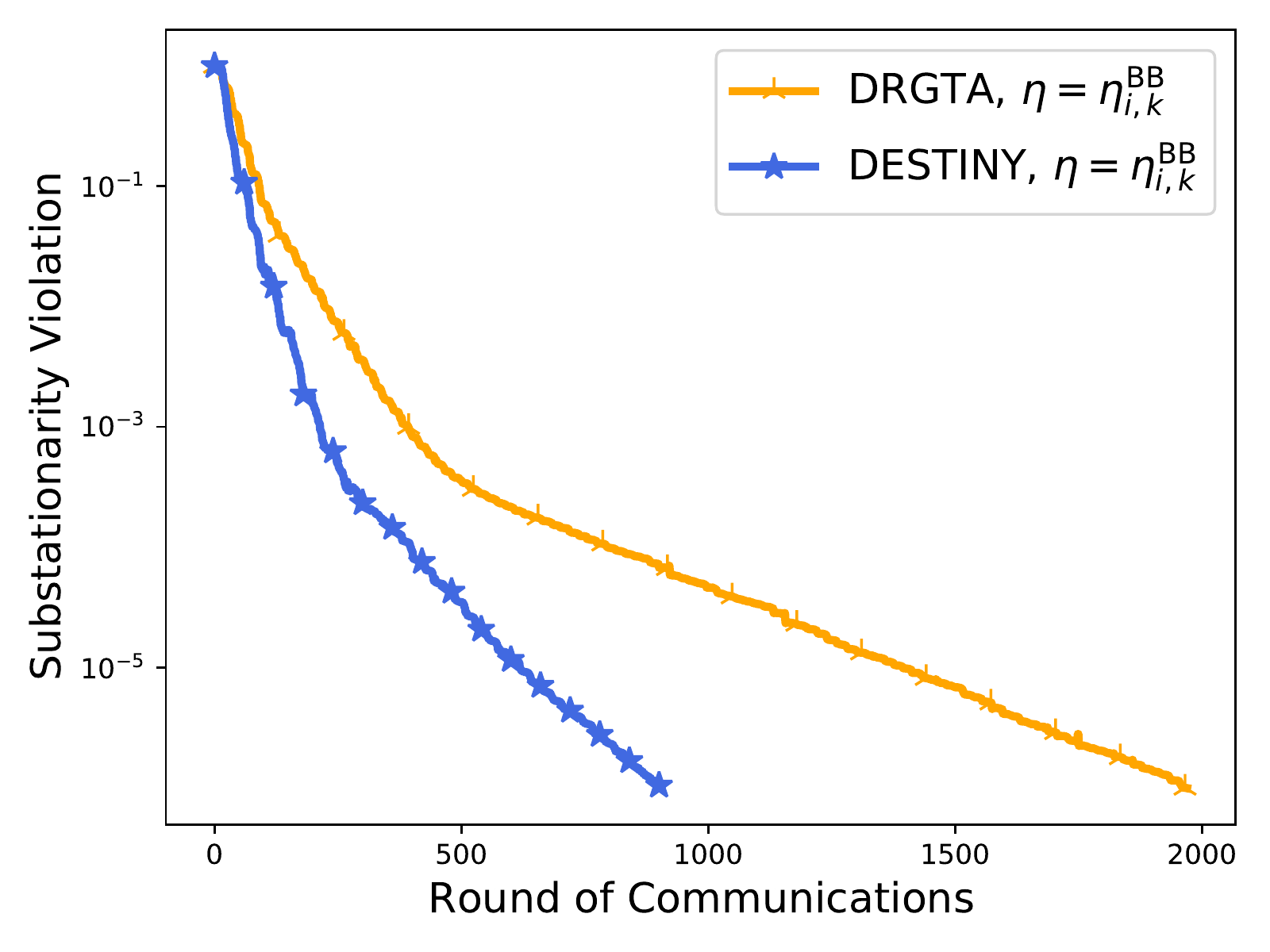}
		}
		\subfigure[Cora-PL]{
			\label{subfig:OLSR_Cora_OS_kkt}
			\includegraphics[width=0.3\linewidth]{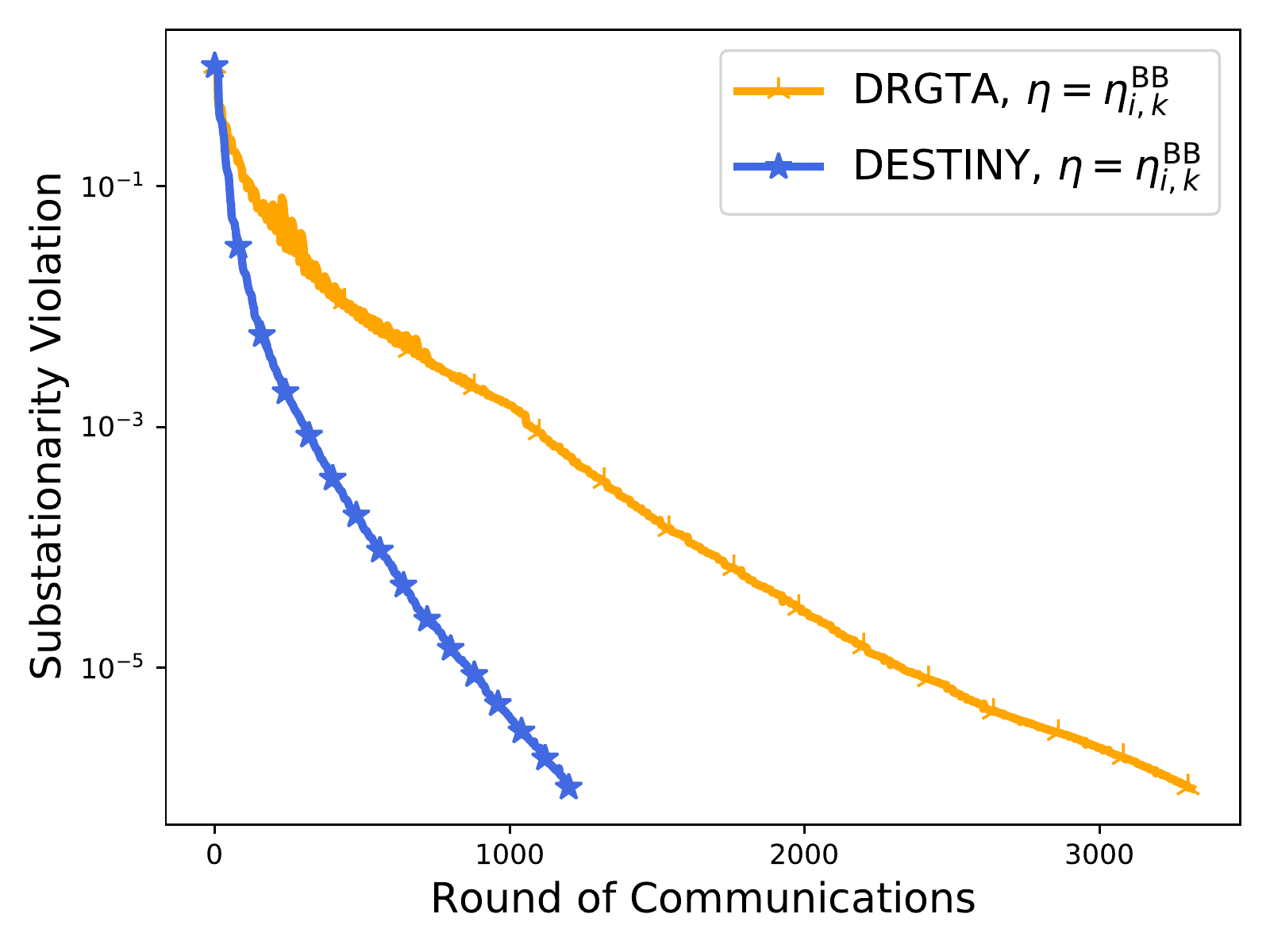}
		}
		
		\subfigure[Cora-DS]{
			\label{subfig:OLSR_Cora_DS_cons}
			\includegraphics[width=0.3\linewidth]{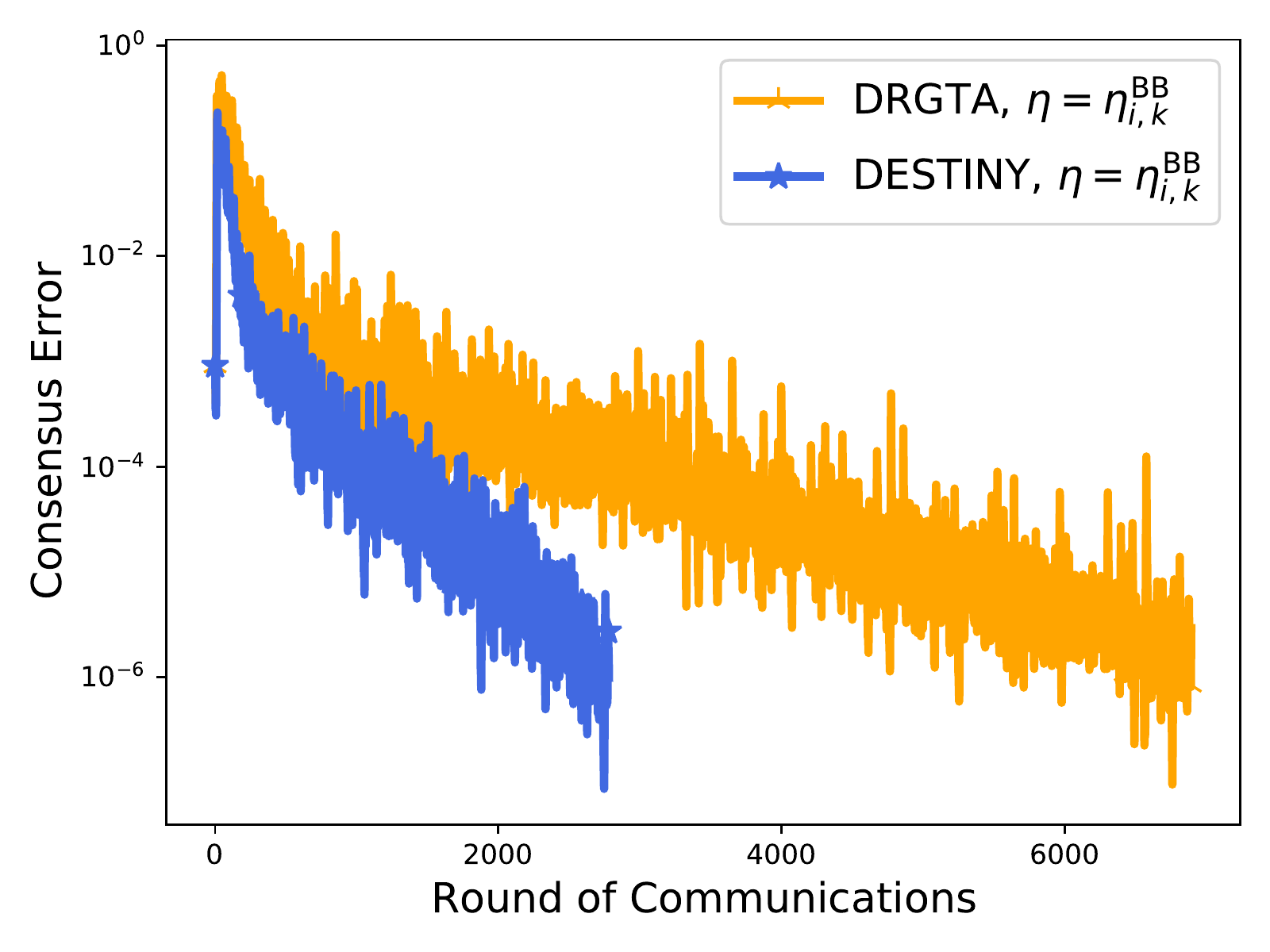}
		}
		\subfigure[Cora-HA]{
			\label{subfig:OLSR_Cora_HA_cons}
			\includegraphics[width=0.3\linewidth]{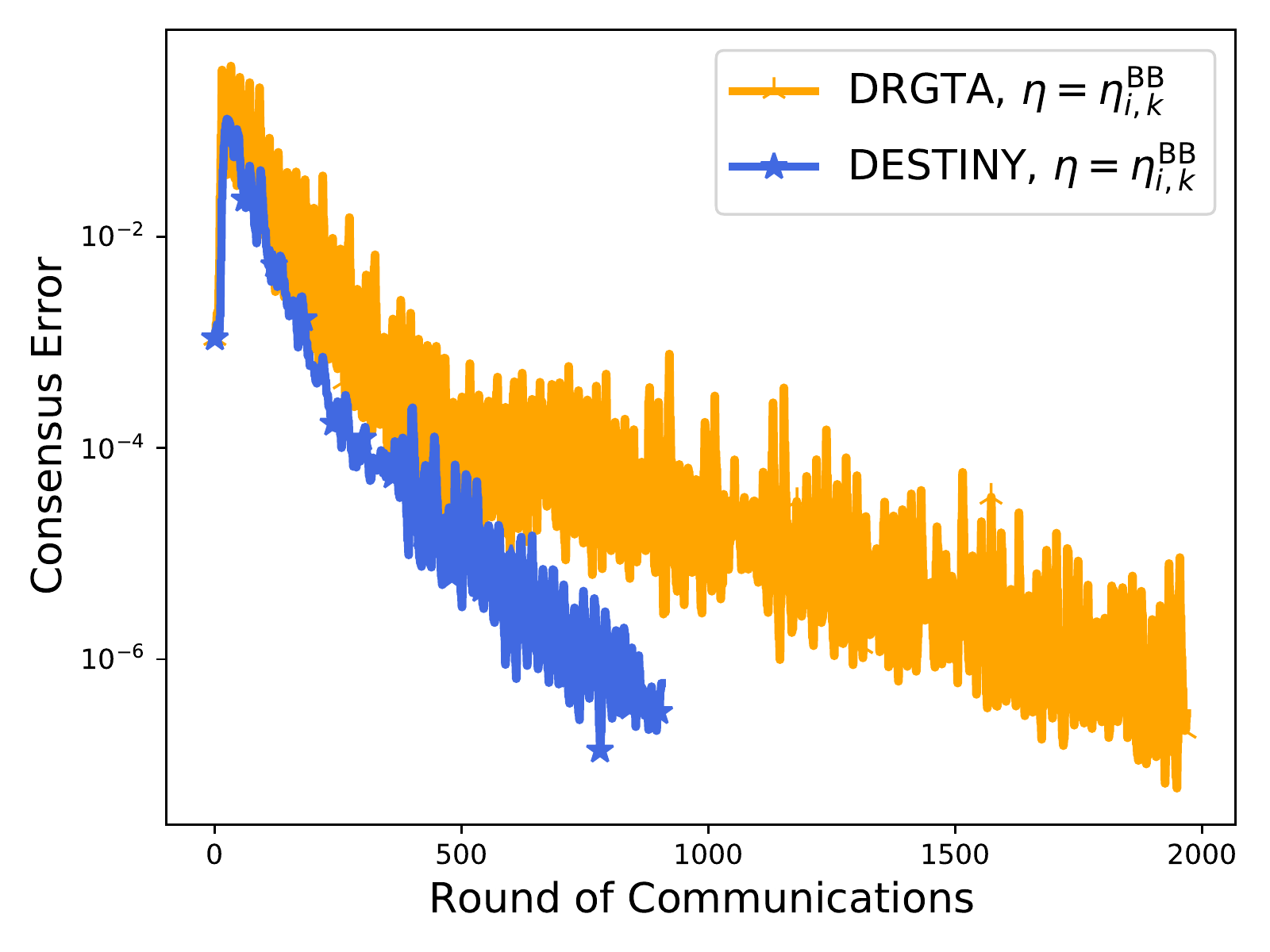}
		}
		\subfigure[Cora-PL]{
			\label{subfig:OLSR_Cora_OS_cons}
			\includegraphics[width=0.3\linewidth]{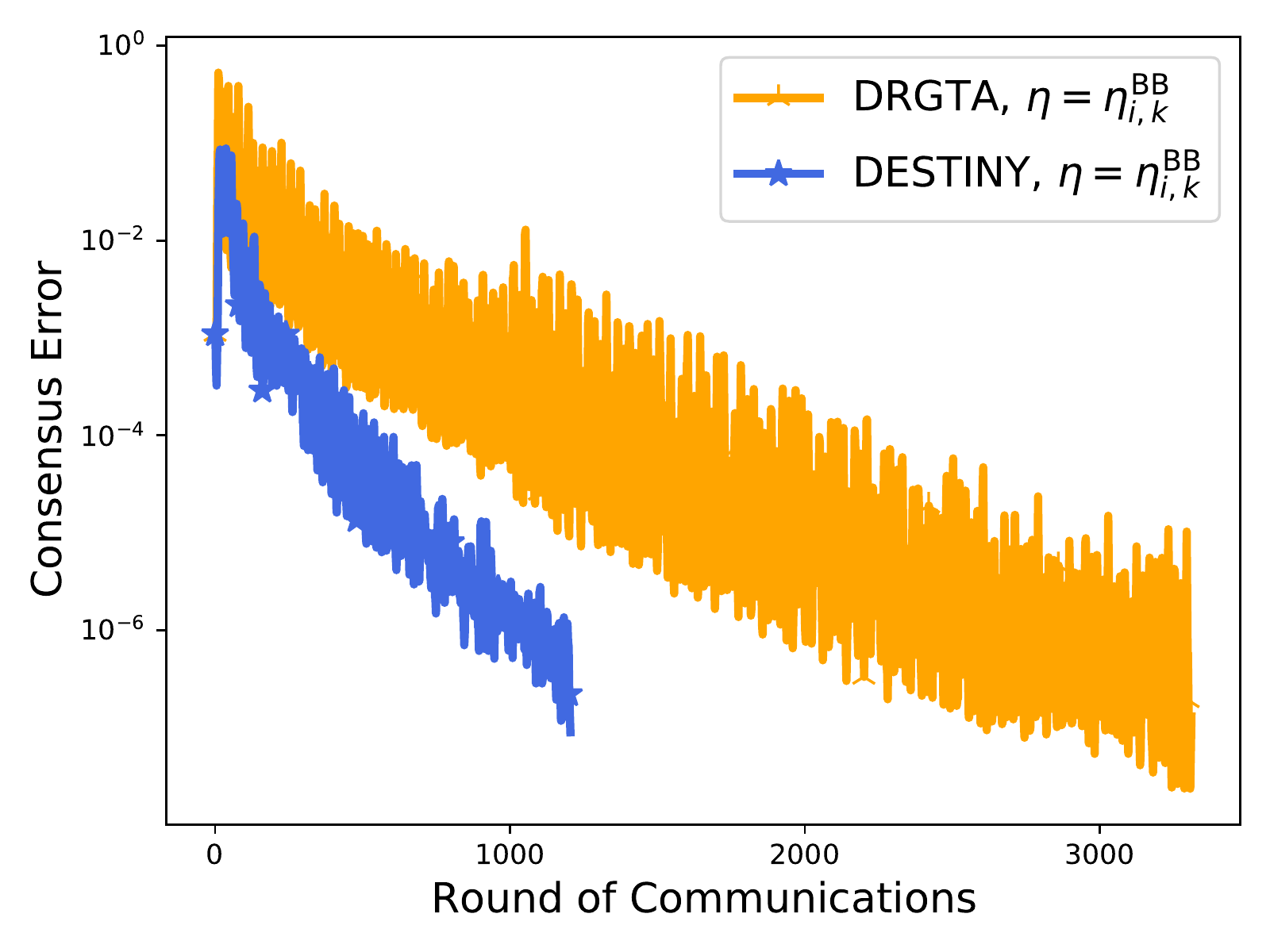}
		}
		
		\caption{Comparison between DRGTA and DESTINY
			for different datasets in solving decentralized OLSR problems. 
			The figures in the first and second rows 
			depict substationarity violations and consensus errors, respectively.}
		\label{fig:OLSR}
	\end{figure}

\subsection{Comparison on decentralized SDL}

	In this subsection, we evaluate the performance on decentralized SDL problems \eqref{eq:opt-sdl}
	of DESTINY and DRGTA \cite{Chen2021decentralized}.
	Three image datasets popular in machine learning research 
	are tested in the following experiments, including
	MNIST\footnote{Available from \url{http://yann.lecun.com/exdb/mnist/} }, 
	CIFAR-10\footnote{Available from \url{https://www.cs.toronto.edu/~kriz/cifar.html}
	\label{note:CIFAR}},
	and CIFAR-100$^{\ref{note:CIFAR}}$.
	We summarize the numbers of features $(n)$ and samples $(m)$ as follows.
	
	\begin{itemize}
		
		\item MNIST: $n = 784$, $m = 60000$.
		
		\item CIFAR-10: $n = 3072$, $m = 50000$.
		
		\item CIFAR-100: $n = 3072$, $m = 50000$.
		
	\end{itemize}
	
	For our testing, the numbers of computed bases and agents 
	are set to $p = 1$ and $d = 32$, respectively.
	Numerical results from this experiment are given in Figure \ref{fig:SDL}.
	It can be observed that DESTINY always dominates DRGTA 
	in terms of both substationarity violations and consensus errors.
	These results indicate that the observed superior performance of our algorithm
	is not just limited to PCA and OLSR problems.
	
	\begin{figure}[ht!]
		\centering
		
		\subfigure[MNIST]{
			\label{subfig:SDL_MNIST_kkt}
			\includegraphics[width=0.3\linewidth]{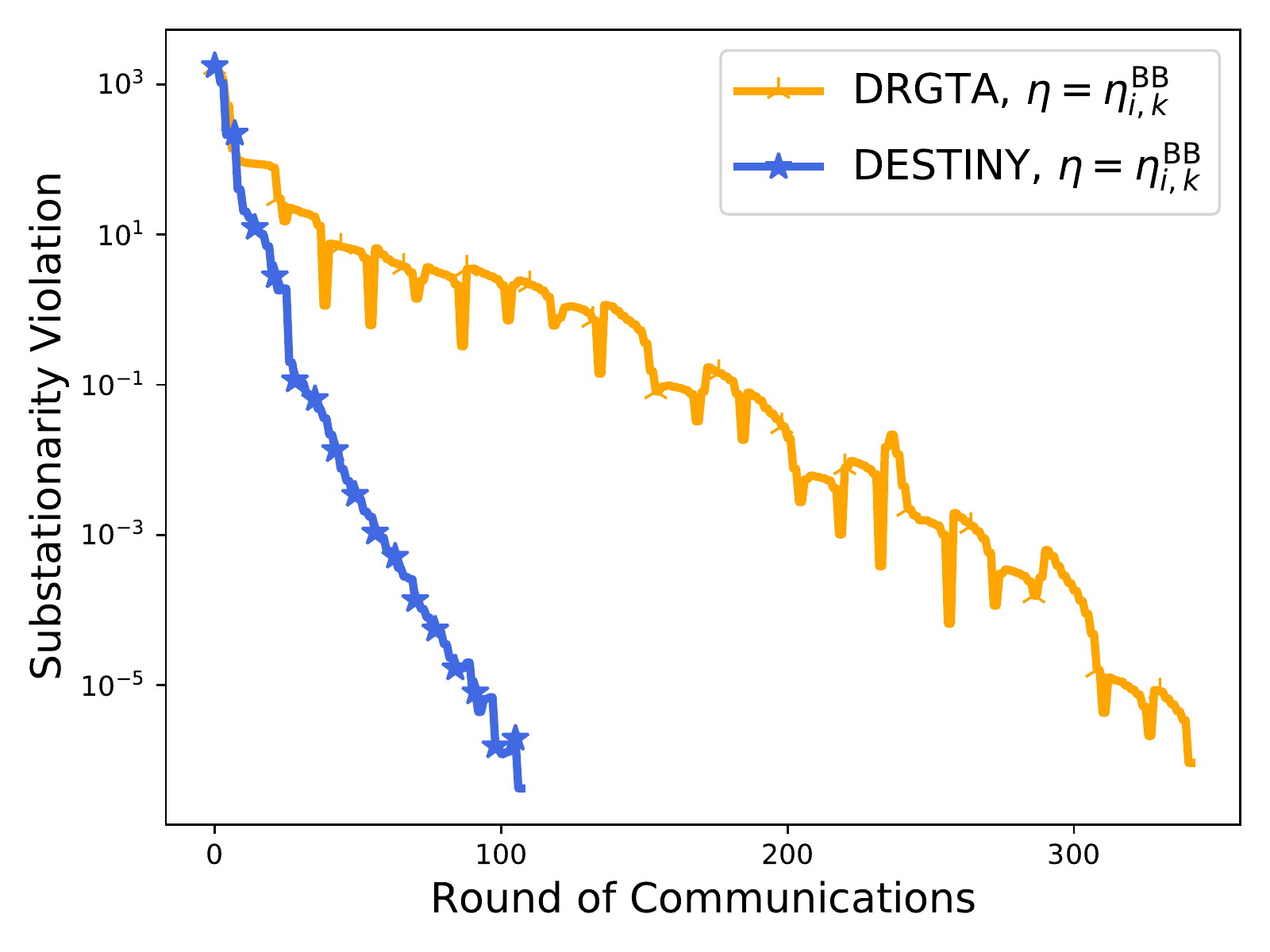}
		}
		\subfigure[CIFAR-10]{
			\label{subfig:SDL_CIFAR-10_kkt}
			\includegraphics[width=0.3\linewidth]{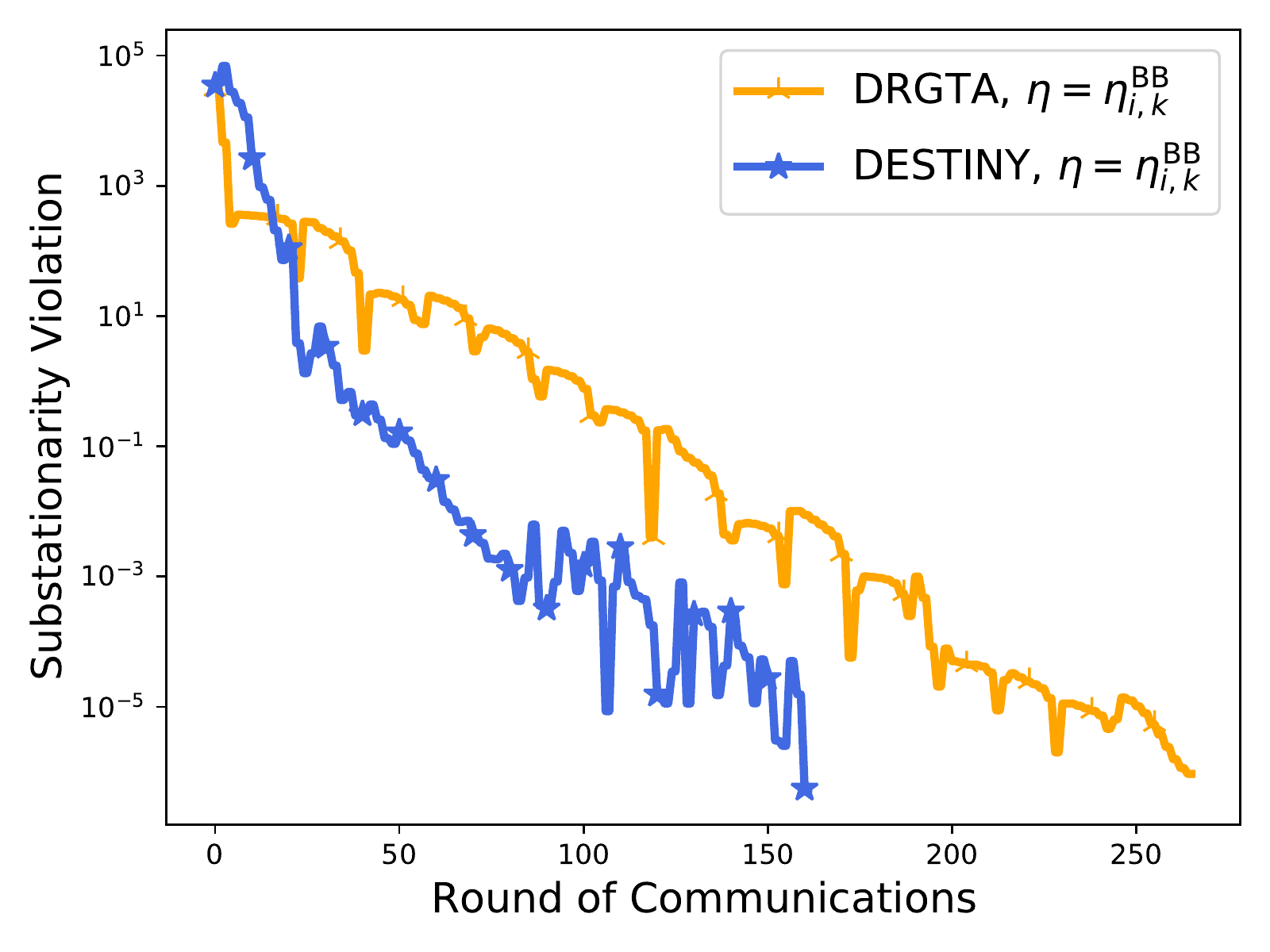}
		}
		\subfigure[CIFAR-100]{
			\label{subfig:SDL_CIFAR-100_kkt}
			\includegraphics[width=0.3\linewidth]{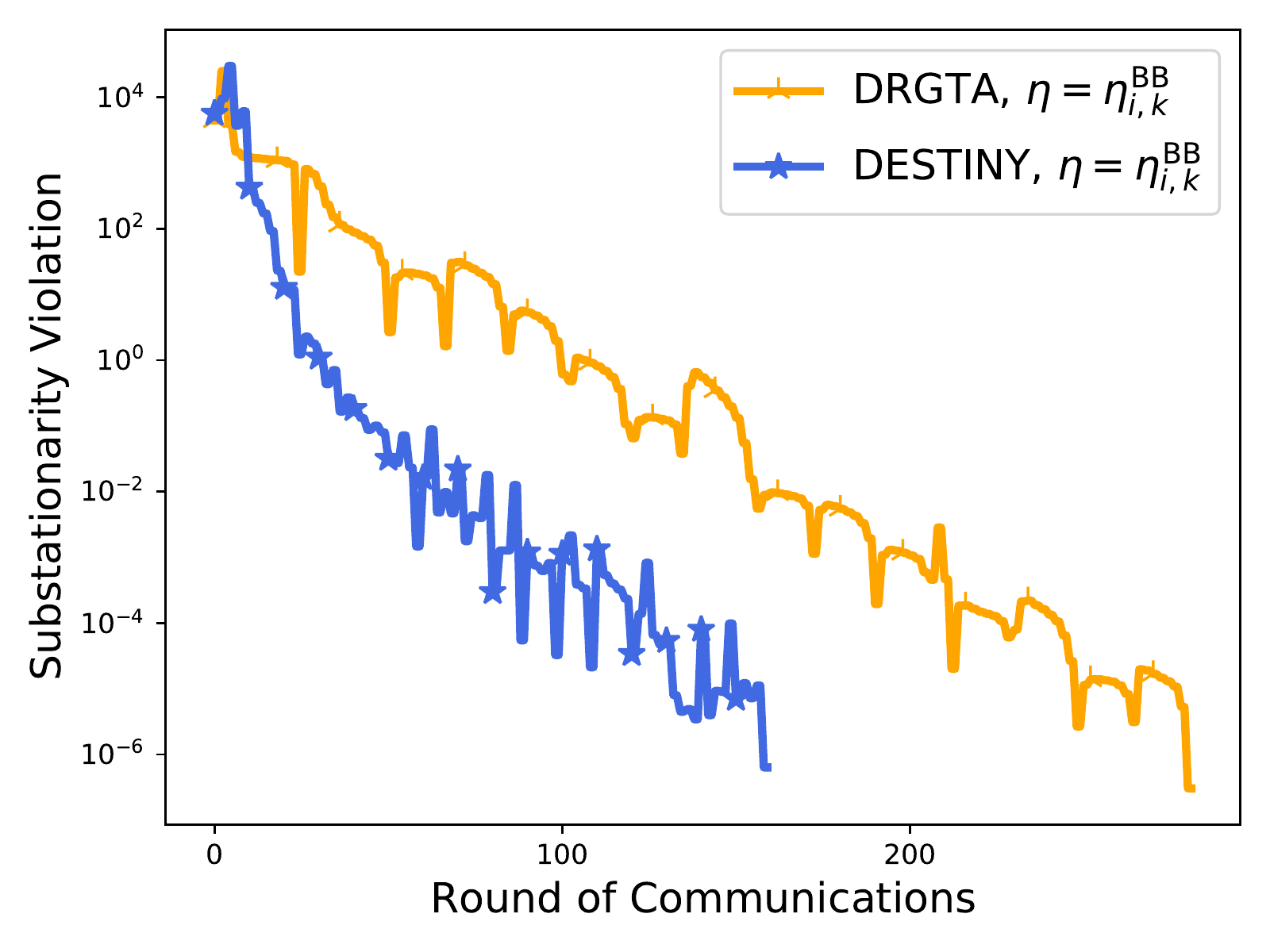}
		}
		
		\subfigure[MNIST]{
			\label{subfig:SDL_MNIST_cons}
			\includegraphics[width=0.3\linewidth]{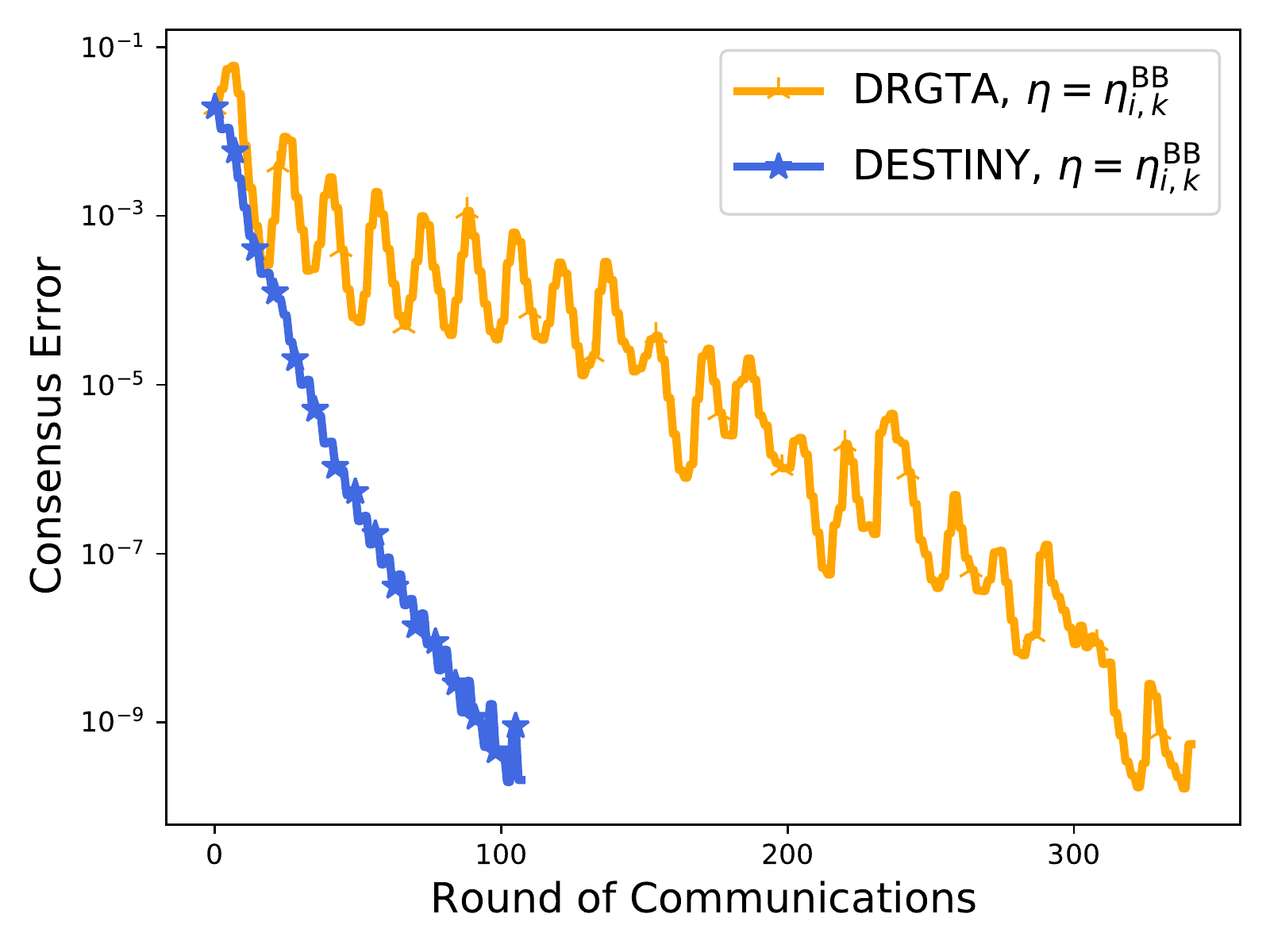}
		}
		\subfigure[CIFAR-10]{
			\label{subfig:SDL_CIFAR-10_cons}
			\includegraphics[width=0.3\linewidth]{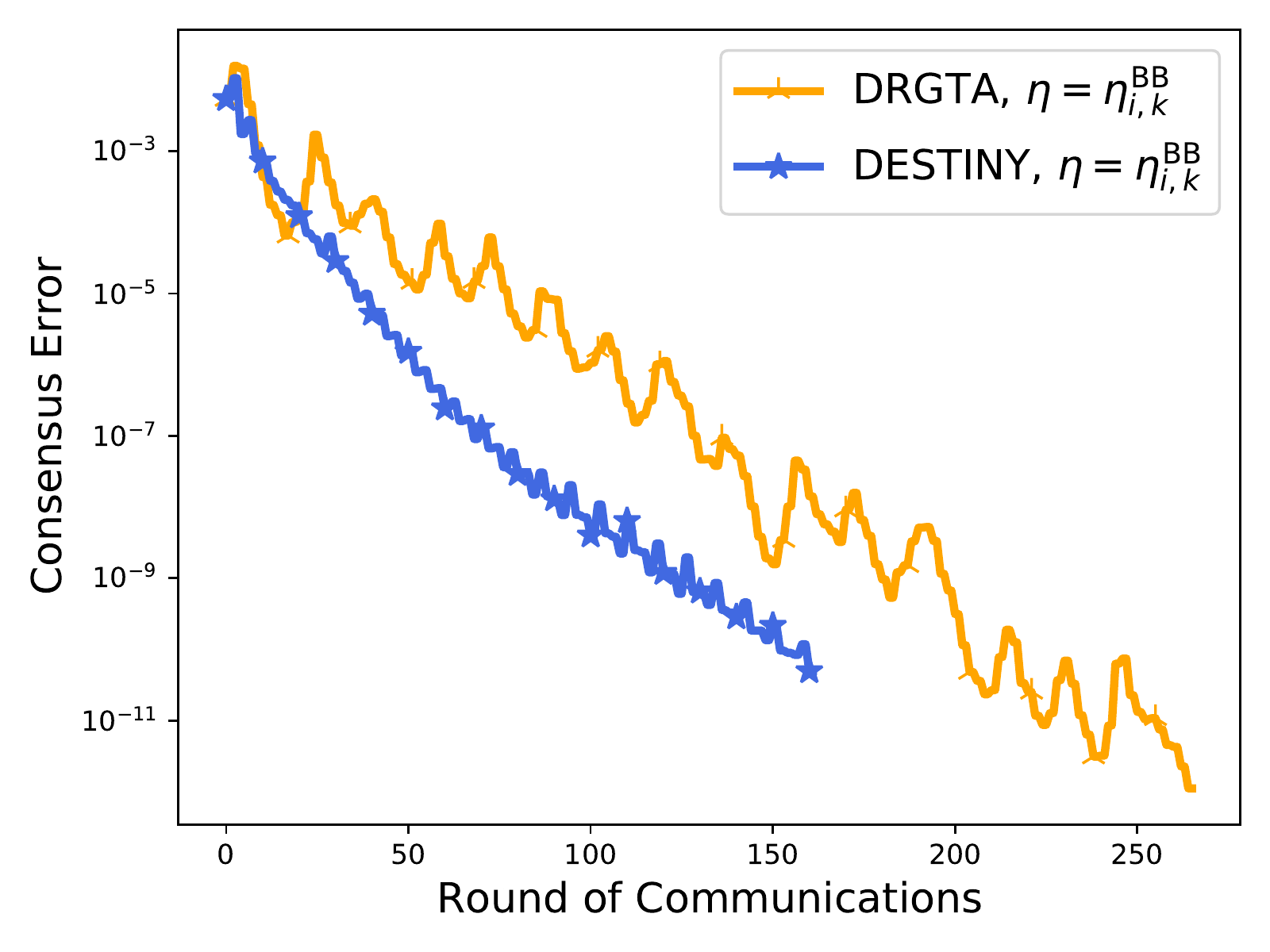}
		}
		\subfigure[CIFAR-100]{
			\label{subfig:SDL_CIFAR-100_cons}
			\includegraphics[width=0.3\linewidth]{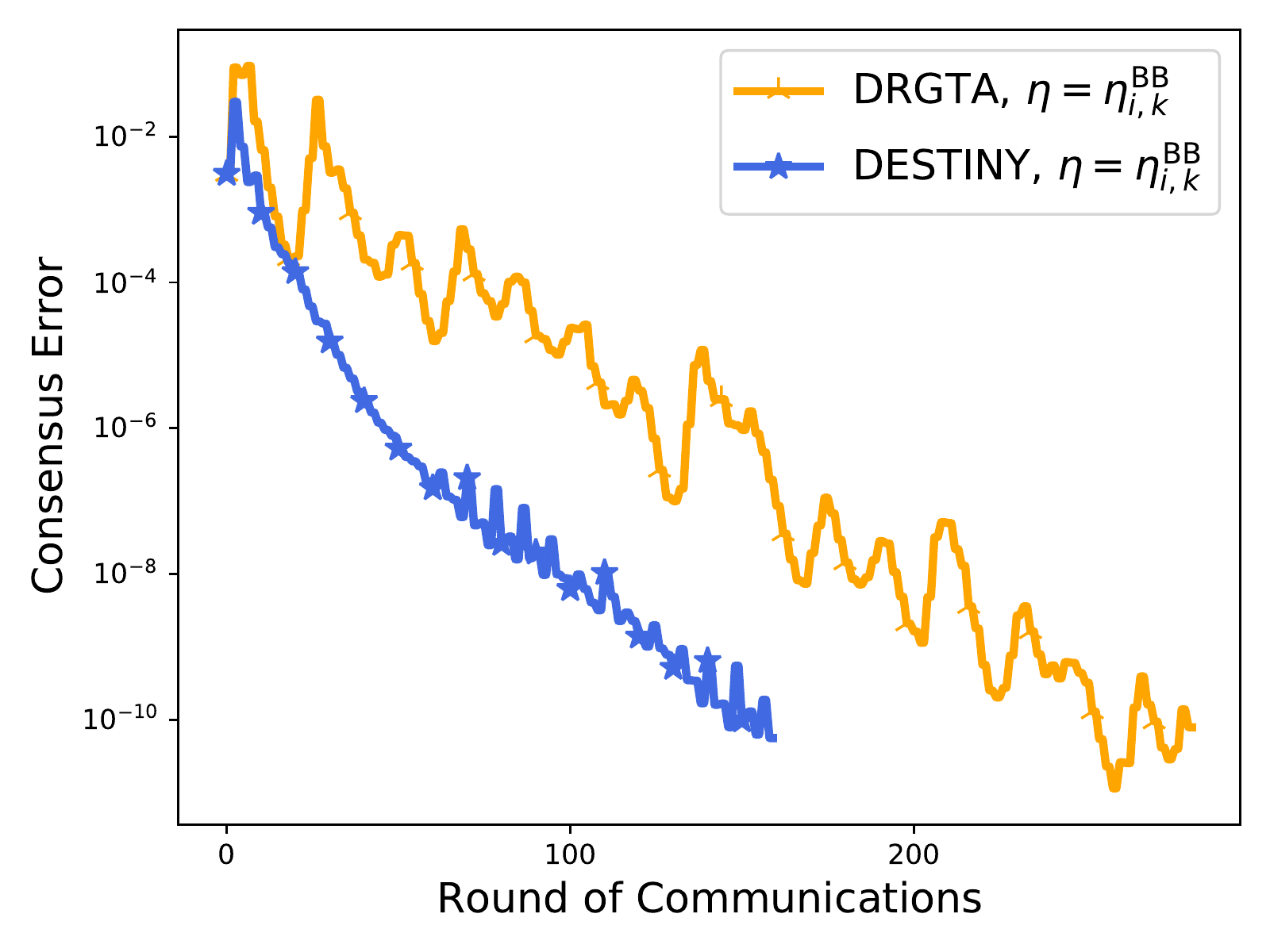}
		}
		
		\caption{Comparison between DRGTA and DESTINY
			for different datasets in solving decentralized SDL problems. 
			The figures in the first and second rows 
			depict substationarity violations and consensus errors, respectively.}
		\label{fig:SDL}
	\end{figure}

\section{Conclusions}\label{eq:conclusion}

	As a powerful and fundamental tool,
	Riemannian optimization has been adapted to 
	solving consensus problems on a multi-agent network,
	giving rise to the decentralized Riemannian gradient tracking algorithm (DRGTA) 
	on the Stiefel manifold.
	In order to guarantee the convergence, 
	DRGTA requires multiple rounds of communications per iteration,
	and hence, the communication overheads are very high.
	In this paper, we propose an infeasible decentralized algorithm, called DESTINY,
	which only invokes a single round of communications per iteration.
	DESTINY seeks a consensus in the Euclidean space directly
	by resorting to an elegant approximate augmented Lagrangian function.
	Even in the non-distributed setting, the proposed algorithm is still new
	for optimization problems over the Stiefel manifold.
	
	Most existing works on analyzing the convergence of gradient tracking based algorithms
	impose stringent assumptions on objective functions,
	such as global Lipschitz smoothness and coerciveness.
	In our case of DESTINY which tackles decentralized optimization problems 
	with nonconvex constraints,
	we are able to derive the global convergence and worst case complexity 
	under rather mild conditions.
	The objective function is only assumed to be locally Lipschitz smooth.
	
	Comprehensive numerical comparisons are carried out under the distributed setting 
	with test results strongly in favor of DESTINY.
	Most notably, the rounds of communications required by DESTINY
	are significantly fewer than what are required by DRGTA.
	
	Finally, some topics are worthy of future studies to fully understand the behavior 
	and realize the potential of DESTINY and its variants.
	For example, the asynchronous version of DESTINY is worthy of investigating
	to address the load balance issues in the distributed environments.
	And more work is needed to adapt DESTINY to the dynamic network settings
	so that it can find a wider range of applications.

\section*{Acknowledgment}

	The authors would like to thank Prof. Lei-Hong Zhang for sharing the Cora dataset,
	and Prof. Qing Ling and Dr. Nachuan Xiao for their insightful discussions.	


\bibliographystyle{siam}

\bibliography{library}

\addcontentsline{toc}{section}{References}

\end{document}